\documentclass[12pt]{article}
\usepackage{graphicx} 
\usepackage{amsmath,amsfonts,amscd,amssymb,color,enumerate,enumitem,empheq,graphicx,framed,multirow,float,amsmath,amssymb,enumerate,amsbsy,amsthm,parskip,color,bbm,amsthm,hyperref}
\usepackage[numbers]{natbib}
\usepackage{enumitem}
\usepackage{listings}   
\usepackage{xcolor}     
\numberwithin{equation}{section}

\lstdefinestyle{MATLABStyle}{
    language=Matlab,                        
    basicstyle=\ttfamily\small,             
    keywordstyle=\color{blue},              
    commentstyle=\color{green!50!black},    
    stringstyle=\color{red},                
    numbers=left,                           
    numberstyle=\tiny\color{gray},          
    stepnumber=1,                           
    numbersep=5pt,                          
    backgroundcolor=\color{white},          
    showspaces=false,                       
    showstringspaces=false,                 
    showtabs=false,                         
    frame=single,                           
    tabsize=2,                              
    captionpos=b,                           
    breaklines=true,                        
    breakatwhitespace=false,                
    escapeinside={\%*}{*)},                 
    morekeywords={*,...}                    
}

\theoremstyle{definition}
\newtheorem{theorem}{Theorem}[section]

\newtheorem{lemma}[theorem]{Lemma}
\newtheorem{definition}{Definition}[section]

\newtheorem{problem}{Problem}[section]
\newcommand{\di}{\displaystyle{}}
\theoremstyle{remark}

\title{Optimal Cosine Polynomials for Riemann Zeta Zero-Free Region}
\author{Hong Sheng Tan \\ \texttt{MAM2309005@XMU.EDU.MY}}  

\date{\today}  

\begin{document}
\maketitle

\section{Introduction}
\label{sec1}

The Riemann zeta function, denoted as $\zeta(s)$ where $s = \sigma + it$, is a fundamental object in analytic number theory, with significant applications in fields such as matrix theory and physics. It is initially defined for $\sigma > 1$ by the series:
\begin{align*}
\zeta(s) = \sum_{n=1}^{\infty} \frac{1}{n^s},
\end{align*}
which converges absolutely in this region. However, for $\sigma \leq 1$,   the series diverges. To address this, $\zeta(s)$ admits an analytic continuation to the entire complex plane, except for a simple pole at $s = 1$. The continuation is expressed as:
\begin{align*}
\zeta(s)=\frac{\pi^{\frac{s}{2}}}{\Gamma\left(\frac{s}{2}\right)}\left(-\frac{1}{s}-\frac{1}{1-s}+\int_1^{\infty} \sum_{n=-1}^{\infty} \exp{(-\pi n^2 x)}\left(x^{s / 2-1}+x^{(1-s) / 2-1}\right) d x\right),
\end{align*}
where $\Gamma(s)$ is the Gamma function.

The zeros of the Riemann zeta function are critical in understanding the distribution of prime numbers. These zeros are classified into trivial zeros, occurring at the negative even integers $s = -2, -4, -6, \ldots$, and nontrivial zeros, which lie within the critical strip $0 < \sigma < 1$. The Riemann Hypothesis conjectures that all nontrivial zeros have a real part $\di\sigma = \frac{1}{2}$. For further proofs and theoretical background on the Riemann zeta function, one can refer to Apostol \cite{ap1976}.

To study the behavior of nontrivial zeros of the Riemann zeta function, the zero-free region of the Riemann zeta function was first established by de la Vallée Poussin \cite{de99} in 1899. The zero-free region is an area in the critical strip where the Riemann zeta function does not vanish. The classical zero-free region has the form:
\begin{align*}
\sigma > 1 - \frac{1}{R \log |t|},
\end{align*}
where $R$ is a positive constant. In this region, $\zeta(s) \neq 0$.

Establishing such regions is crucial, as it directly influences the understanding of the distribution of prime numbers. The prime number counting function
\begin{align*}
\pi(x) = \sum_{p \leq x} 1,
\end{align*}
counts the number of primes less than $x$. Landau \cite{la08} proved that for any $\rho > R$:
\begin{align*}
\pi(x) = \int_{2}^{x} \frac{dy}{\log y} + O\left(x (\log x)^{-1/2} \exp{\left(-\sqrt{\log x / \rho}\right)}\right).
\end{align*}

A smaller $R$ will give a larger zero-free region and a smaller error term in the prime number counting function. Efforts to enlarge the zero-free region have driven researchers to reduce the constant $R$ using various techniques. These refinements improve the approximation of the prime number counting function and yield deeper insights into the zeta function’s zeros. Table~\ref{tab:1} summarizes key results in reducing the $R$ constant.

\begin{table}[h!]
\centering
\begin{tabular}{|c|c|c|}
\hline
Year & Author & $R$ \\
\hline
1899 & de la Vallée Poussin \cite{de99}& 30.47 \\
1962 & Rosser, Schoenfeld \cite{ro62} & 17.52 \\
1970 & Stechkin \cite{st70}& 9.65 \\
1975 & Rosser, Schoenfeld \cite{ro75}& 9.65 \\
2005 & Kadiri \cite{ka05}& 5.69 \\
2022 & Mossinghoff, Trudgian, Yang \cite{mo22}& 5.56 \\
\hline
\end{tabular}
\caption{Values of the $R$ constant determined by existing studies.}
\label{tab:1}
\end{table}

All of this work has employed a very special type of trigonometric polynomial. Given a positive integer $n\geq 2$, let $C_n$ denotes the set of cosine polynomials:
\begin{align*}
f(\phi) = \sum_{k=0}^{n} a_k \cos(k\phi),
\end{align*}
with the following properties:
\begin{enumerate}[label=A.\arabic*, ref=a.\arabic*]
    \item All the coefficients are nonnegative and real.
    \item The coefficient $a_1 > a_0 > 0$.
    \item $f(\phi) \geq 0$ for all real $\phi$.
\end{enumerate}
Then, the real valued functional $v: C_n \to \mathbb{R}$ is defined by:
\begin{align*}
v(f) = \frac{f(0) - a_0}{(\sqrt{a_1} - \sqrt{a_0})^2}.
\end{align*}

Landau \cite{la08} demonstrated that for any $f \in C_n$, one can take:
\begin{align*}
R = \frac{v(f)}{2},
\end{align*}
in the classical zero-free region, ensuring $\zeta(s) \neq 0$ in this region.

Naturally, the goal to enlarge the zero-free region and pursuits the smaller $R$ leads to the following problems.

\begin{problem} What is the exact value of
\begin{align*}
V_n = \inf_{f \in C_n} \frac{f(0) - a_0}{(\sqrt{a_1} - \sqrt{a_0})^2},
\end{align*}
and
\begin{align*}
V = \inf_{n \geq 2} V_n.
\end{align*}
\end{problem}
The problems to find the value of $V$ were studied by Westphal \cite{we38}, Stechkin \cite{st70}, Rosser and Schoenfeld \cite{ro62,ro75},  French \cite{fr66}, Kondrat'ev \cite{ko77,ar90}, Reztsov \cite{re86}, Arestov \cite{ar90, ar92}, Mossinghoff and Trudgian \cite{mo14}. Most of the works are dedicated to estimating the upper and lower bounds of $V$. As of now, the best two-sided estimates for $V$ are:
\begin{align*}
34.468305 < V < 34.503586.
\end{align*}

The exact value of $V_n$ is known only for $n = 2, 3, 4, 5, 6$.

To determine the exact value of $V_2$, French \cite{fr66} substituted $x=\cos\phi$ and expressed every cosine polynomial in $C_2$ in the form:
\begin{align*}
h(x) = (1 - a_2) + a_1 x + 2 a_2 x^2,
\end{align*}
which satisfies the condition $ h(x)\geq 0$ for all $x \in [-1, 1]$. French classified the polynomial in two classes based on the global minimum point $(x_0, h(x_0))$ of $h(x)$. The first class is when $x_0 \in [-1, 1]$ and the second class is when $x_0 \not\in [-1, 1]$. Then, French calculated the minimum value of $v(f)$ can be attained in two different classes, the smaller number being the exact value of $V_2$.

On the other hand, Arestov use the following properties of nonnegative trigonometric polynomials to evaluate the exact value of $V_n$ for $n=3,4,5,6$. Firstly, he noticed that for any trigonometric polynomial $f_n \in C_n$, the Fejér inequality \cite{po78} holds:
\begin{align}\label{eq1.1}
a_1 (f) \leq A(n) a_0 (f), \quad A(n) = 2 \cos \frac{\pi}{n+2}.
\end{align}
Secondly, for any $f\in C_n$ and positive number $k$, one has $v(kf)=v(f)$. By using the homogeneity of $v(f)$ and Fejér inequality, one can change the problem of finding $V_n$ to become:
\begin{align}\label{eq1.2}
V_n = \inf_{a \in (1, A(n)]} \frac{\chi_n(a)}{(\sqrt{a} - 1)^2},
\end{align}
where $\chi_n(a)$ is defined as:
\begin{align*}
\chi_n(a) = \inf \{ f(0) - 1 \mid f(\phi)\in C_n, a_0(f) = 1 \text{ and } a_1(f) = a \}.
\end{align*}

Arestov \cite{ar92} used some inequalities of nonnegative trigonometric polynomials to find the exact value of $\chi_n(a)$ on some specific intervals. This enabled him to calculate the exact value of $V_n$ for $n = 3, 4, 5, 6.$ The values of $V_n$ for $n = 2, 3, 4, 5, 6$ are summarized in Table~\ref{tab:2}.

\begin{table}[h!]
\centering
\begin{tabular}{|c|c|c|}
\hline
n & Author & $V_n$ \\
\hline
2 & French \cite{fr66} & 53.1390720  \\
3 & Arestov \cite{ar92} & 36.9199911  \\
4 & Arestov \cite{ar92} & 34.8992259 \\
5 & Arestov \cite{ar92} & 34.8992259 \\
6 & Arestov \cite{ar92} & 34.8992259 \\
\hline
\end{tabular}
\caption{Values of $V_n$ as determined by French and Arestov.}
\label{tab:2}
\end{table}

The goal of this research is to revise the exact values of $V_n$ for $n = 2, 3, 4, 5, 6$ and find the exact values of $V_n$ for $n=7,8$. In Section 2, the exact values of $V_2$ and $V_3$ will be found by improving French’s method. Instead of classifying the polynomial $h(x)$ by the location of the minimum point, the polynomial will be classified based on the location of its roots.

Sections 3 to Section 8 will be devoted to finding the exact values of $V_n$ for $n$ from 4 to 8. By applying the Fejér-Riesz theorem \cite{fe16}, the problem of finding $\chi_n(a)$ is transformed into an optimization problem in $\mathbb{R}^n$ and then solved using optimization techniques. The techniques and general procedure will be described in Section 3, while the key constraints and result of the optimization problems will be summarized in Section 4 to Section 8. In the last two sections, the following theorems will be presented.

\textbf{Theorem 1.1}: The exact value of $V_7$ is
\begin{align*}
V_7 = 34.6494874.
\end{align*}

\textbf{Theorem 1.2}: The exact values of $V_8$ is
\begin{align*}
V_8=34.5399155.
\end{align*}

This work aims to advance our understanding of the zero-free regions of the Riemann zeta function, thereby contributing to the broader field of analytic number theory and its applications to prime number theory.

\vfil\pagebreak
\section{The exact value of $V_2$ and $V_3$}

\subsection{Some properties of linear functional $v(f)$.}
This section is devoted to finding the exact values of $V_2$ and $V_3$. Firstly, some useful properties of the function $v(f)$ will be established. These properties will enable a focus on some special forms of polynomials to find the values of $V_2$ and $V_3$.

One particularly useful property is the homogeneity of the function $v(f)$, which implies that for $k>0$, the function $v$ satisfies $v(kf)=v(f)$. This property is advantageous because, instead of studying all polynomials $f(\phi) \in C_n$, the focus can be narrowed to polynomials of the form
\begin{align*}
f(\phi)=1+a \cos \phi+a_2 \cos 2 \phi+\ldots+a_n \cos n \phi,
\end{align*}
where $\di a \in\left(1,2 \cos \frac{\pi}{n+2}\right]$ by Fejér's Inequality (Equation (\ref{eq1.1})).

To explore next property of $v(f)$, the following two theorems imply that one only need to consider the cosine polynomial with a minimum of 0 in order to find the exact value of $V_n$.
\begin{theorem}\label{th2.1} 
Given
\begin{align*}
f(\phi)=a_0+a_1 \cos \phi+a_2 \cos 2 \phi+\ldots+a_n \cos n \phi \in C_n,
\end{align*}
let $m$ be the minimum of $f(\phi)$ when $\phi \in[0,2 \pi)$, then
\begin{align*}
m< a_0.
\end{align*}
\end{theorem}

\begin{proof} This can be proven easily by considering the integration below,
$$
a_0=\frac{1}{2 \pi} \int_0^{2 \pi} f(\phi) d \phi \geq \frac{1}{2\pi}(2\pi) m=m.
$$

The equality holds if and only if $f(\phi)$ is a constant function, which is impossible in the context of this discussion. This completes the proof.
\end{proof}
The next theorem states that given a polynomial $f(\phi) \in C_n$ with minimum $m>0$, another polynomial $\bar{f}(\phi) \in C_n$ with minimum 0 can always be found such that $v(f)>v(\bar{f})$. This clearly implies $v(f)>V_n$, thus $f(\phi)$ is not a candidate for the infimum.
\begin{theorem}\label{th2.2}
Given $f(\phi) \in C_n$ with minimum $m>0$, let
$$
\bar{f}(\phi)=f(\phi)-m .
$$

Then $\bar{f}(\phi) \in C_n$ and
$$
v(f)>v(\bar{f}).
$$
\end{theorem}

\begin{proof}  To prove $\bar{f}(\phi) \in C_n$, note that $a_0(\bar{f})=a_0(f)-m>0$ by Theorem \ref{th2.1}. Moreover, it holds that
$$
a_1(\bar{f})=a_1(f)>a_0(f)>a_0(\bar{f}).
$$

Lastly,
\begin{align*}
v(\bar{f})&=\frac{a_1(f)+a_2(f)+a_3(f)+\ldots+a_n(f)}{\left(\sqrt{a_1(f)}-\sqrt{a_0(f)-m}\right)^2}\\
&<\frac{a_1(f)+a_2(f)+a_3(f)+\ldots+a_n(f)}{\left(\sqrt{a_1(f)}-\sqrt{a_0(f)}\right)^2}\\
&=v(f).
\end{align*}

This completes the proof.

\end{proof}
With the previous theorems, it is established that only cosine polynomials with a minimum of 0 need to be considered in the process of evaluating exact value of $V_n$. The next step is to analyze the properties of the root of these cosine polynomials.
\begin{theorem}\label{th2.3}
Given a non-negative cosine polynomial $f(\phi) \in C_n$ with minimum 0 , where
$$
f(\phi)=a_0+a_1 \cos \phi+a_2 \cos 2 \phi+\ldots+a_n \cos n \phi,
$$
then $f(\phi)$ must have at least one root $\phi_0$ in the interval $[0,2 \pi)$, and it must satisfy one of the following:
\begin{enumerate}[label=\alph*., ref=\alph*.]
\item The root $\phi_0=\pi$ (or $\cos\phi_0=-1$).
\item The root must have even order.
\end{enumerate}
This will imply if $\phi_0\neq \pi$ (or $\cos\phi_0 \neq -1$), then it must have even order.
\end{theorem}
\begin{proof}
By substituting $x=\cos \phi, f(\phi)$ can be transformed into the polynomial form:
$$
h(x)=b_n x^n+\ldots+b_2 x^2+b_1 x+b_0,
$$
where the fact that $f(\phi)$ is non-negative and has a minimum of zero implies $h(x)$ is non-negative in the interval $[-1,1]$ and has at least one root in this interval. If the root of $h(x)$ lies within the interval $(-1,1)$, it must be of even order, as it also functions as a local minimum point for the function $h(x)$. In other words, only $x=1$ and $x=-1$ can serve as the root in $[-1,1]$ with odd order.

However, $x=1$ cannot serve as the root of $h(x)$. This is shown by contradiction. If $h(x)$ has $x=1$ as its root, then $f(\phi)$ can be expressed as
$$
f(\phi)=(\cos \phi-1) g_{n-1}(\phi),
$$
where $g_{n-1}(\phi)$ is a nonnegative trigonometric polynomial with degree $n-1$, this implies $f(0)=0$. This will never happen since
$$
f(0)=a_n+a_{n-1}+\ldots+a_2+a_1+a_0>0.
$$

This completes the proof.
\end{proof}

Finally, to find the exact value of $V_n$, it is enough to consider only cosine polynomials with at least one double root. Using a similar method in the proof of Theorem \ref{th2.2}, it will be shown that given any cosine polynomial $f(\phi) \in C_n$ which does not have a double root, a polynomial $\bar{f}(\phi) \in C_n$ with a double root can always be found such that $v(f)>v(\bar{f})$. This implies $f(\phi)$ is not a candidate for attaining infimum.

The exact formula of $\bar{f}(\phi)$ will be provided in Theorem \ref{th2.4}, and subsequently, in Theorem \ref{th2.5}, it will be shown that $\bar{f}$ satisfies the rquire properties.
\begin{theorem}\label{th2.4}
Given
$$
\begin{aligned}
& f(\phi)=(\cos \phi+1)\left(f_{n-1}(\phi)+m\right), \\
& \bar{f}(\phi)=(\cos \phi+1) f_{n-1}(\phi),
\end{aligned}
$$
where $m>0$ and $f_{n-1}(\phi)$ is an $(n-1)$-degree nonnegative trigonometric polynomial with minimum 0, then
\begin{enumerate}[label=\alph*., ref=\alph*.]
\item The polynomial $f(\phi) \in C_n$ if and only if $\bar{f}(\phi) \in C_n$.
\item If $f(\phi) \in C_n$, then
$$
v(f(\phi))>v(\bar{f}(\phi)) .
$$
\end{enumerate}
\end{theorem}

\begin{proof}
First of all, one can easily see that both $f(\phi)$ and $\bar{f}(\phi)$ satisfy the nonnegative condition. To prove the statement (a), note that
$$
f(\phi)=\bar{f}(\phi)+m \cos \phi+m.
$$

If $f(\phi) \in C_n$, then all the coefficients of $\bar{f}(\phi)$ should also be nonnegative since $a_k(\bar{f}) \geq a_k(f)$ for $k=1,2, \ldots, n$. Lastly, it is also easy to check $a_1(\bar{f})-a_0(\bar{f})=a_1(f)-$ $a_0(f)>0$, which implies $\bar{f}(\phi) \in C_n$.

Conversely, assume that $\bar{f}(\phi) \in C_n$. It will be shown that $f(\phi) \in C_n$. Firstly, by Theorem \ref{th2.1}, $a_0(f)>0$. The fact that $a_1(\bar{f})-a_0(\bar{f})=a_1(f)-a_0(f)>0$ implies
$$
a_1(f)>a_0(f)>0
$$

Moreover, $a_n(f) \geq 0$ for $n \geq 2$ can be deduced by the fact $a_n(f)=a_n(\bar{f}) \geq 0$ for $n \geq 2$.

To prove the second statement, just note that for $m>0$, it holds that
$$
\left(\sqrt{a_1(\bar{f})+m}-\sqrt{a_0(\bar{f})+m}\right)^2<\left(\sqrt{a_1(\bar{f})}-\sqrt{a_0(\bar{f})}\right)^2
$$

This implies
$$
v(f(\phi))=\frac{a_3(\bar{f})+a_2(\bar{f})+a_1(\bar{f})+m}{\left(\sqrt{a_1(\bar{f})+m}-\sqrt{a_0(\bar{f})+m}\right)^2}>\frac{a_3(\bar{f})+a_2(\bar{f})+a_1(\bar{f})}{\left(\sqrt{a_1(\bar{f})}-\sqrt{a_0(\bar{f})}\right)^2}=v(\bar{f}(\phi))
$$

This completes the proof.

\end{proof}

To complete the discussion of this subsection, the final theorem will be proven, which shows the $\bar{f}(\phi)$ has at least one double root.
\begin{theorem}\label{th2.5}
Given $f(\phi) \in C_n$ which does not have a double root, it is always possible to find another $\bar{f}(\phi) \in C_n$ with at least one double root such that
$$
v(f)>v(\bar{f}).
$$

\end{theorem}

\begin{proof}
One can always use Theorem \ref{th2.2} to replace $f(\phi)$ with a cosine polynomial that has at least one root. Thus, without loss of generality, only the cosine polynomial with at least a single root in the interval $[0,2 \pi)$ needs to be considered. By Theorem \ref{th2.3}, the cosine polynomial $f(\phi)$ must have the form
$$
f(\phi)=(\cos \phi+1)\left(f_{n-1}(\phi)+m\right)
$$
where $f_{n-1}(\phi)$ is a $(n-1)$-degree nonnegative trigonometric polynomial with minimum 0 and $m>0$. Then, let
$$
\bar{f}(\phi)=(\cos \phi+1) f_{n-1}(\phi)
$$
Theorem \ref{th2.4} implies that $\bar{f}(\phi) \in C_n$ and $v(f)>v(\bar{f})$. The last thing to prove is that $\bar{f}(\phi)$ has at least one double root. Let $x=\cos \phi$ and transform $\bar{f}(\phi)$ to $\bar{h}(x)$, that is $\bar{h}(\cos \phi)=\bar{f}(\phi)$. Then
$$
\bar{h}(x)=(x+1) h_{n-1}(x)
$$
where $h_{n-1}(x)$ is nonnegative in the interval $[-1,1]$, and the minimum in this interval is 0 . The polynomial $h_{n-1}(x)$ has a zero in the interval $[-1,1]$. If the zero is in $(-1,1)$, then it is a double zero. If the zero is $x=-1$, then $\bar{h}(x)$ has a double zero at $x=-1$. If the zero is $x=1, then f(0)=h(1)=0$ implies $\bar{f}(\phi)=\bar{h}(\cos \phi) \notin C_n$, which implies $f(\phi) \notin C_n$, which is a contradiction. The proof is complete.
\end{proof}

As a summary, to find the exact value of $V_n$, it is only necessary to consider the cosine polynomials in the following set.
\begin{definition}\label{df2.1}
The set $C_n^0$ is the set containing all degree $n$ cosine polynomials
$$
f(\phi)=a_0+a_1 \cos \phi+a_2 \cos 2 \phi+\ldots+a_n \cos n \phi
$$
which satisfy the following three conditions:
\begin{enumerate}[label=B.\arabic*, ref=a.\arabic*]
\item The function $f(\phi)$ is nonnegative and has a double root.
\item All the coefficients are nonnegative.
\item The coefficient $a_1>a_0>0$.
\end{enumerate} 
\end{definition}
By focusing on the special forms of polynomials and utilizing the properties discussed in the theorems above, the process of finding the exact values of $V_2$ and $V_3$ can be simplified. 
\subsection{Exact value of $V_2$.}
To find the exact value of $V_2$, Theorem \ref{th2.5} and Definition \ref{df2.1} imply that it is only necessary to consider the cosine polynomial with double root. Moreover, the homogeneity of $v(f)$ implies it is enough to consider cosine polynomial with the form
\[
f(\phi) = (\cos \phi + \alpha)^2,
\]
where $\alpha \in [-1, 1]$. However, one still need to check when $f(\phi)\in C_2^0$.

\begin{theorem}\label{th2.6}
For $-1 \leq \alpha \leq 1$, the polynomial
\[
f(\phi) = (\cos \phi + \alpha)^2 \in C_2^0
\]
if and only if $\di 1 - \frac{1}{\sqrt{2}} < \alpha < 1$.
\end{theorem}

\begin{proof}
The nonnegative property of $f(\phi)$ is obvious. Expanding $f(\phi)$, we have
\[
f(\phi) = \cos^2 \phi + 2\alpha \cos \phi + \alpha^2 = \cos 2\phi + 2\alpha \cos \phi + \alpha^2 + 1.
\]
Then $f(\phi) \in C_2^0$ if and only if
\[
2\alpha > \alpha^2 + \frac{1}{2}.
\]
Rearranging, we get
\[
\alpha^2 - 2\alpha + \frac{1}{2} < 0,
\]
which simplifies to
\[
1 - \frac{1}{\sqrt{2}} < \alpha < 1.
\]
This completes the proof.
\end{proof}
The next theorem presents the exact value of $V_2$ and the cosine polynomial $f(\phi)\in C_2$ such that $v(f)=V_2$. 
\begin{theorem}\label{th2.7}
The exact value of $V_2$ is
\[
V_2 = 53.1390720,
\]
which can be obtained by considering
\[
f(\phi) = (\cos \phi + 0.7415574)^2.
\]
\end{theorem}

\begin{proof}
By Theorem \ref{th2.6}, only the polynomials in the following form need to be considered:
\[
f(\phi) = (\cos \phi + \alpha)^2,
\]
where $\di 1 - \frac{1}{\sqrt{2}} < \alpha < 1$. In this case, 
\[
f(\phi) = \cos^2 \phi + 2\alpha \cos \phi + \alpha^2 = \frac{1}{2}\cos 2\phi + 2\alpha \cos \phi + \alpha^2 + \frac{1}{2},
\]
thus
\[
v(f(\phi)) = \frac{0.5 + 2\alpha}{(\sqrt{2\alpha} - \sqrt{\alpha^2 + 0.5})^2}.
\]
This function is continuous in the interval $\di \alpha \in \left(1 - \frac{1}{\sqrt{2}}, 1\right)$. Using Matlab to conduct a golden section search, the minimum is found to be 53.1390720 (see Figure \ref{fig1}),
\begin{figure}[h]
\centering
\includegraphics[width=9cm]{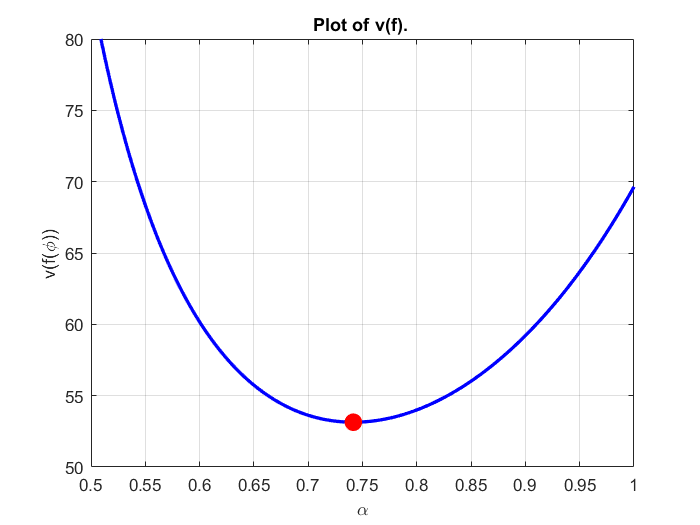}
\caption{Graph of $v(f(\phi))$.}\label{fig1}
\end{figure}
\\\\which can be obtained by considering
\[
f(\phi) = (\cos \phi + 0.7415574)^2.
\]
This completes the proof.
\end{proof}

\subsection{Exact value of $V_3$.}
This subsection will be devoted to finding the exact value of $V_3$. To find the exact value of $V_3$, Theorem \ref{th2.5} and Definition \ref{df2.1} imply that it is only necessary to consider the cosine polynomial with double root. Moreover, the homogeneity of $v(f)$ implies it is enough to consider cosine polynomial with the form
\[
f(\phi) = (\cos \phi + \alpha)^2(\cos \phi + \beta),
\]
where $\alpha \in [-1, 1]$ and $\beta \geq 1$.

Our next goal is to identify all the polynomials in this form that belong to $C_3^0$.

\begin{theorem}\label{2.8}
For $-1 \leq \alpha \leq 1$ and $\beta \geq 1$, the cosine polynomial
\[
f(\phi) = (\cos \phi + \alpha)^2(\cos \phi + \beta)
\]
belongs to $C_3^0$ if and only if
\begin{enumerate}[label=\alph*., ref=\alph*.]
    \item $\di\alpha \geq 1 - \frac{1}{\sqrt{2}}$, or
    \item $\di -\frac{1}{4} < \alpha < 1 - \frac{1}{\sqrt{2}}$ and $\di \beta \in \left[1, 1 + \frac{4\alpha + 1}{4\alpha^2 - 8\alpha + 2}\right)$.
\end{enumerate}
\end{theorem}

\begin{proof}
To check that a cosine polynomial belongs to $C_3^0$, we need to check all the conditions stated in Definition \ref{df2.1}. It is obvious that for every $\alpha \in [-1, 1]$ and $\beta \geq 1$, the polynomial $f(\phi)$ is nonnegative and has a double root. Next, expanding $f(\phi)$, we obtain
\[
f(\phi) = \cos 3\phi + (2\beta + 4\alpha) \cos 2\phi + (4\alpha^2 + 8\alpha\beta + 3) \cos \phi + (4\alpha^2\beta + 4\alpha + 2\beta).
\]

To ensure $f(\phi)$ satisfies $a_1 > a_0$, consider the inequality:
\[
4\alpha^2 + 8\alpha\beta + 3 > 4\alpha^2\beta + 4\alpha + 2\beta.
\]
Rearranging, one get
\[
4\alpha^2 - 4\alpha + 3 > \beta(4\alpha^2 - 8\alpha + 2).
\]
It is obvious that $4\alpha^2 - 4\alpha + 3 = (2\alpha - 1)^2 + 2$ is positive. When $\di\alpha \geq 1 - \frac{1}{\sqrt{2}}$, we have $\di 4\alpha^2 - 8\alpha + 2 \leq 0$, thus the inequality holds. When $\di\alpha < 1 - \frac{1}{\sqrt{2}}$, one needs to have
\[
\beta < \frac{4\alpha^2 - 4\alpha + 3}{4\alpha^2 - 8\alpha + 2} = 1 + \frac{4\alpha + 1}{4\alpha^2 - 8\alpha + 2}.
\]

Thus, when $\di\alpha \leq -\frac{1}{4}$, one has $\di\beta < 1$, which contradicts the assumption. When $\di -\frac{1}{4} < \alpha < 1 - \frac{1}{\sqrt{2}}$, one has
\[
\beta \in \left[1, 1 + \frac{4\alpha + 1}{4\alpha^2 - 8\alpha + 2}\right).
\]

One can also check that when $\di\alpha > -\frac{1}{4}$, we have $a_2 = 4\alpha + 2\beta > 0$ and $a_0 = 4\alpha^2\beta + 4\alpha + 2\beta > 0$. This shows the polynomial $f(\phi)$ is in $C_3^0$. This completes the proof.
\end{proof}

Next, it will be demonstrated that in order to determine the exact value of $V_3$, it suffices to consider degree 3 polynomials of the form
\[
f(\phi) = (\cos \phi + \alpha)^2(\cos \phi + 1),
\]
where $\di\alpha \in \left(-\frac{1}{4}, 1\right]$.
The same strategy as used in the proof of Theorem \ref{th2.2} will be applied here. Given a cosine polynomial $f(\phi)$ in the form
\[
f(\phi) = (\cos \phi + \alpha)^2(\cos \phi + \beta),
\]
where $\di\alpha \in \left(-\frac{1}{4}, 1\right]$ and $\beta > 1$, another polynomial
\[
\bar{f}(\phi) = (\cos \phi + \alpha)^2(\cos \phi + 1)
\]
 can be constructed and yields a better result, implying that $v(f)>v(\bar{f})$. Consequently, $f$ is not a candidate for the infimum. 

To establish this result, we first require some preliminary lemmas. Three of the following lemmas provide simple inequalities that will be used to compare the values of $v(f)$ and $v(\bar{f})$.  Readers may choose to skip ahead to Theorem \ref{th2.12} and return to these lemmas when the inequalities are needed. For a smoother reading experience, the proofs of Lemma \ref{th2.9}, Lemma \ref{th2.10} and Lemma \ref{th2.11} are presented in the \ref{Appendix A}.
\begin{lemma}\label{th2.9}
Given four positive real numbers $a, b, \Delta a, \Delta b$, if
\[
\frac{\Delta a}{\Delta b} > \frac{a}{b},
\]
then
\[
\frac{a + \Delta a}{b + \Delta b} > \frac{a}{b}.
\]
\end{lemma}

\begin{lemma} \label{th2.10}
Given five positive real numbers $a_1, a_2, b_1, b_2, k$, if
\[
\frac{a_1}{b_1} > k \quad \text{and} \quad \frac{a_2}{b_2} > k,
\]
then
\[
\frac{a_1 + a_2}{b_1 + b_2} > k.
\]
\end{lemma}

\begin{lemma}\label{th2.11}
Given four positive real numbers $a, b, \Delta a, \Delta b$, then
\[
\sqrt{(a + \Delta a)(b + \Delta b)} \geq \sqrt{ab} + \sqrt{\Delta a \Delta b},
\]
equality holds when $a \Delta b = b \Delta a.$
\end{lemma}

The following theorem compares the value of $v(f)$ and $v(\bar{f})$, and implies that we only need to consider the cosine polynomial in the form of  
\[
\bar{f}(\phi) = (\cos \phi + \alpha)^2(\cos \phi + 1)
\]
in order to find the exact value of $V_3$.
\begin{theorem}\label{th2.12}
Given a cosine polynomial in $C_3^0$ of the form
\[
f(\phi) = (\cos \phi + \alpha)^2(\cos \phi + \beta),
\]
let
\[
\bar{f}(\phi) = (\cos \phi + \alpha)^2(\cos \phi + 1).
\]
Then $\bar{f}(\phi) \in C_3^0$ and
\[
v(f(\phi)) > v(\bar{f}(\phi)).
\]
\end{theorem}

\begin{proof}
When $\di\alpha \in \left(-\frac{1}{4}, 1 - \frac{1}{\sqrt{2}}\right]$, it will be shown that $v(f(\phi))$ increases as $\beta$ increases by proving $\di\frac{\partial v(f)}{\partial \beta} > 0$ for $\beta \geq 1$. Expanding $f(\phi)$, we have
\[
f(\phi) = \cos 3\phi + (2\beta + 4\alpha) \cos 2\phi + (4\alpha^2 + 8\alpha\beta + 3) \cos \phi + (4\alpha^2\beta + 4\alpha + 2\beta),
\]
thus
\[
v(f(\phi)) = \frac{4\alpha^2 + 8\alpha\beta + 4\alpha + 2\beta + 4}{\left(\sqrt{4\alpha^2 + 8\alpha\beta + 3} - \sqrt{4\alpha^2\beta + 4\alpha + 2\beta}\right)^2}.
\]
To show $\di\frac{\partial v(f)}{\partial \beta}$ is positive, it is enough to prove
\[
\begin{aligned}
\bar{v} &= \left(\sqrt{4\alpha^2 + 8\alpha\beta + 3} - \sqrt{4\alpha^2\beta + 4\alpha + 2\beta}\right)(8\alpha + 2) \\
& - \left(4\alpha^2 + 8\alpha\beta + 4\alpha + 2\beta + 4\right)\left(\frac{8\alpha}{\sqrt{4\alpha^2 + 8\alpha\beta + 3}} - \frac{4\alpha^2 + 2}{\sqrt{4\alpha^2\beta + 4\alpha + 2\beta}}\right) > 0.
\end{aligned}
\]

When $\di\alpha \in \left(-\frac{1}{4}, 1 - \frac{1}{\sqrt{2}}\right]$, we have $4\alpha^2 + 2 \geq 8\alpha$, and thus
\[
\frac{8\alpha}{\sqrt{4\alpha^2 + 8\alpha\beta + 3}} \leq \frac{4\alpha^2 + 2}{\sqrt{4\alpha^2 + 8\alpha\beta + 3}} \leq \frac{4\alpha^2 + 2}{\sqrt{4\alpha^2\beta + 4\alpha + 2\beta}}.
\]

This implies
\[
\frac{8\alpha}{\sqrt{4\alpha^2 + 8\alpha\beta + 3}} - \frac{4\alpha^2 + 2}{\sqrt{4\alpha^2\beta + 4\alpha + 2\beta}} \leq 0,
\]
and thus $\bar{v} > 0$ since we know that $\sqrt{a_1} - \sqrt{a_0} > 0$, $8\alpha + 2 > 0$, and $a_3 + a_2 + a_1 > 0$.

For the case where $\di\alpha \geq 1 - \frac{1}{\sqrt{2}}$, we want to prove that
\[
v(f(\phi)) > v(\bar{f}(\phi)).
\]

If we write $f(\phi) = (\cos \phi + \alpha)^2(\cos \phi + 1 + m)$, where $m > 0$, then the statement we want to prove turns to
\[
\begin{aligned}
& \frac{4\alpha^2 + 12\alpha + 6 + m(8\alpha + 2)}{\left(\sqrt{4\alpha^2 + 8\alpha + 3 + 8\alpha m} - \sqrt{4\alpha^2 + 4\alpha + 2 + 2m + 4\alpha^2 m}\right)^2} \\
> & \frac{4\alpha^2 + 12\alpha + 6}{\left(\sqrt{4\alpha^2 + 8\alpha + 3} - \sqrt{4\alpha^2 + 4\alpha + 2}\right)^2}.
\end{aligned}
\]

Now, let
\[
\begin{aligned}
& \Delta_1 = (8\alpha + 2)m, \\
& \Delta_2 = \left(\sqrt{a_1 + 8\alpha m} - \sqrt{a_0 + (4\alpha^2 + 2)m}\right)^2 - \left(\sqrt{a_1} - \sqrt{a_0}\right)^2,
\end{aligned}
\]
where $a_1 = 4\alpha^2 + 8\alpha + 3$ and $a_0 = 4\alpha^2 + 4\alpha + 2$. By Lemma \ref{th2.11}, one can show that
\[
\begin{aligned}
\Delta_2 & = \left(4\alpha^2 + 8\alpha + 2\right)m - 2\sqrt{\left(a_1 + 8\alpha m\right)\left(a_0 + \left(4\alpha^2 + 2\right)m\right)} + 2\sqrt{a_1 a_0} \\
& < \left(4\alpha^2 + 8\alpha + 2\right)m - 2\sqrt{a_1 a_0} - 2m\sqrt{8\alpha\left(4\alpha^2 + 2\right)} + 2\sqrt{a_1 a_0} \\
& = m\left[4\alpha^2 + 8\alpha + 2 - 2\sqrt{8\alpha\left(4\alpha^2 + 2\right)}\right].
\end{aligned}
\]

Using Matlab to conduct a golden section search (see Figure \ref{fig2}), one can check that for $\di\alpha \in \left[1 - \frac{1}{\sqrt{2}}, 1\right]$,
\[
\frac{\Delta_1}{\Delta_2} > \frac{8\alpha + 2}{4\alpha^2 + 8\alpha + 2 - 2\sqrt{8\alpha\left(4\alpha^2 + 2\right)}} \geq 53.1390720.
\]
\begin{figure}[h]
\centering
\includegraphics[width=9cm]{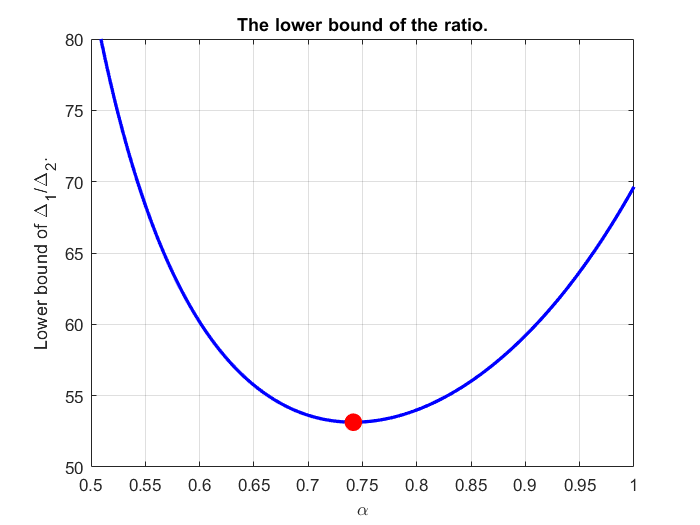}
\caption{The lower bound of $\frac{\Delta_1}{\Delta_2}$.}\label{fig2}
\end{figure}
The minimum point can be obtained when $\alpha = 0.7415574$.

Moreover, one can also use the same method (see Figure \ref{fig3}) to check 
\[
v(\bar{f}(\phi)) = \frac{4\alpha^2 + 12\alpha + 6}{\left(\sqrt{4\alpha^2 + 8\alpha + 3} - \sqrt{4\alpha^2 + 4\alpha + 2}\right)^2} \leq 22 + \frac{44\sqrt{6}}{5} = 43.5555097.
\]
\begin{figure}[h]
\centering
\includegraphics[width=9cm]{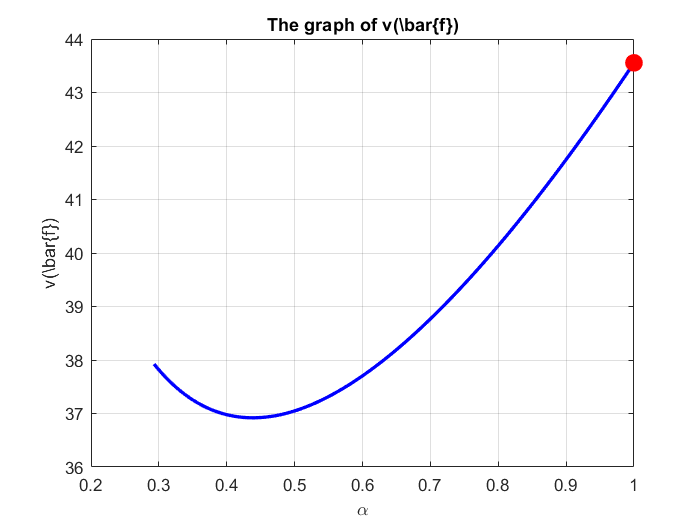}
\caption{Graph of $v(\bar{f} (\phi))$.}\label{fig3}
\end{figure}
The maximum point can be obtained when $\alpha = 1$. Thus, by Lemma \ref{th2.9}, we have
\[
v(f(\phi)) = \frac{4\alpha^2 + 12\alpha + 6 + \Delta_1}{\left(\sqrt{4\alpha^2 + 8\alpha + 3} - \sqrt{4\alpha^2 + 4\alpha + 2}\right)^2 + \Delta_2} > v(\bar{f}(\phi)).
\]

This completes the proof.
\end{proof}

Finally, the last theorem of this section presents the exact value of $V_3$ and identifies the cosine polynomial $f(\phi)\in C_3$ such that $v(f)=V_3$. 

\begin{theorem}
The exact value of $V_3$ is
\[
V_3 = 36.9199911,
\]
which can be obtained by considering the cosine polynomial
\[
f(\phi) = 4(\cos \phi + 1)(\cos \phi + 0.4384345)^2.
\]
\end{theorem}

\begin{proof}
Theorem \ref{2.8} implies that we only need to consider the cosine polynomial in the form
\[
f(\phi) = 4(\cos \phi + \alpha)^2(\cos \phi + \beta)
\]
where $\alpha, \beta$ satisfy one of the following:
\begin{enumerate}[label=\alph*., ref=\alph*.]
    \item $\di \alpha \in \left[1 - \frac{1}{\sqrt{2}}, 1\right]$, or
    \item $\di -\frac{1}{4} < \alpha < 1 - \frac{1}{\sqrt{2}}$ and $\beta \in \left[1, 1 + \frac{4\alpha + 1}{4\alpha^2 - 8\alpha + 2}\right)$.
\end{enumerate}
Moreover, Theorem \ref{th2.12} implies that we need only consider the case where $\beta = 1$. As a result, the polynomial takes the form
\[
f(\phi) = 4(\cos \phi + \alpha)^2(\cos \phi + 1),
\]
where $\di\alpha \in \left(-\frac{1}{4}, 1\right]$. Expanding $f(\phi)$, we obtain
\[
f(\phi) = \cos 3\phi + (4\alpha + 2) \cos 2\phi + (4\alpha^2 + 8\alpha + 3) \cos \phi + (4\alpha^2 + 4\alpha + 2).
\]
In this case, we have
\[
v(f(\phi)) = \frac{4\alpha^2 + 12\alpha + 6}{\left(\sqrt{4\alpha^2 + 8\alpha + 3} - \sqrt{4\alpha^2 + 4\alpha + 2}\right)^2}.
\]

One can check that this function is continuous in the interval $\di \left(-\frac{1}{4}, 1\right]$. Using Matlab to conduct a golden section search (see Figure \ref{fig4}), the minimum is 36.9199911, which can be obtained by considering the cosine polynomial
\[
f(\phi) = 4(\cos \phi + 1)(\cos \phi + 0.4384345)^2.
\]
\begin{figure}[h]
\centering
\includegraphics[width=9cm]{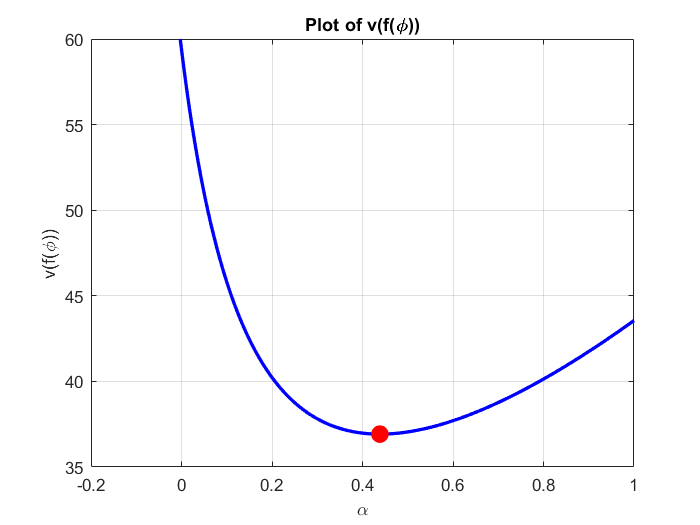}
\caption{Graph of $v(f(\phi))$.}\label{fig4}
\end{figure}
This completes the proof.
\end{proof}

\vfil\pagebreak
\section{Methodology} 
The Fejer-Riesz theorem was first employed by Kondrat'ev \cite{ko77} to find the upper bound of $V_8$ in his paper. This method was later employed by Mossinghoff and Trudgian \cite{mo14} to determine the upper bound of $V_n$ for $n = 8, 12, 16, \ldots, 40$. The Fejer-Riesz theorem translates the problem of finding the exact value of $V_n$ into a minimization problem on $\mathbb{R}^n$, subject to certain conditions. The statement of the Fejer-Riesz theorem is presented as follows:

\begin{theorem}[Fejer-Riesz theorem]\cite{fe16}\label{th3.1}
Let $\displaystyle g(\phi) = \sum_{k=0}^n b_k \cos k \phi$ be a trigonometric polynomial with real coefficients, which is nonnegative for every real $\phi$ and $b_n \neq 0$. Then, there exist real numbers $\alpha_0, \alpha_1, \alpha_2, \ldots, \alpha_n$ such that
\[
g(\phi) = \left| \sum_{k=0}^n \alpha_k \exp{(i k \phi)} \right|^2.
\]

Moreover, we also have
\[
b_0 = \sum_{k=0}^n \alpha_k^2 \quad \text{and} \quad b_j = \sum_{k=0}^{n-j} \alpha_k \alpha_{k+j}.
\]
\end{theorem}

Now, given any cosine polynomial in $C_n$, 
\[
f(\phi) = 1 + a \cos \phi + \sum_{k=2}^n a_k \cos k \phi ,
\]
the Fejer-Riesz theorem states that we can find real numbers $x_0, x_1, x_2, \ldots, x_n$ such that
\[
f(\phi) = \left| \sum_{k=0}^n x_k \exp{(i k \phi)} \right|^2,
\]
with the following conditions:
\begin{align*}
&x_0^2 + x_1^2 + x_2^2 + \ldots + x_n^2 = 1, \\
&2(x_0 x_1 + x_1 x_2 + \ldots + x_{n-1} x_n) = a, \\
&2(x_0 x_k + x_1 x_{k-1} +  \ldots +x_{n-k} x_n) \geq 0,
\end{align*}
for $2 \leq k \leq n$. 

By observing that
$$f(0)=(x_0 + x_1 + \ldots + x_n)^2 - 1,$$
the problem of finding the exact value of $\chi_n(a)$ in Equation (\ref{eq1.2}) is equivalent to the following minimization problem, which we call Problem $P_1$.

\begin{problem}\label{p3.1} The optimization Problem $P_1$ is defined as
\[
\begin{aligned}
\text{$P_1$: Minimize } & F(\mathbf{x}) = (x_0 + x_1 + \ldots + x_n)^2 - 1 \\
\text{subject to } & H_1(\mathbf{x}) = x_0^2 + x_1^2 + x_2^2 + \ldots + x_n^2 = 1, \\
& H_2(\mathbf{x}) = 2(x_0 x_1 + x_1 x_2 + \ldots + x_{n-1} x_n) = a, \\
& G_1(\mathbf{x}) = 2(x_0 x_2 + x_1 x_3 + \ldots + x_{n-2} x_n) \geq 0, \\
& G_2(\mathbf{x}) = 2(x_0 x_3 + x_1 x_4 + \ldots + x_{n-3} x_n) \geq 0, \\
& \quad\quad\quad\vdots \\
& G_{n-1}(\mathbf{x}) = 2x_0 x_n \geq 0.
\end{aligned}
\]
\end{problem}
For Problem $P_1$, we denote the subset of $\mathbb{R}^n$ that satisfies all the conditions in the statement of the problem as $S_1$. One can note that if $S_1$ is not empty, then it will be a closed set since
\[
\text{S1} = H_1^{-1}(\{1\}) \cap H_2^{-1}(\{a\}) \cap \bigcap_{i=1}^{n-1} G_i^{-1}([0, \infty)).
\]
Moreover, the condition $H_1(\mathbf{x}) = 1$ implies that $S_1$ is bounded, meaning $S_1$ is a compact set. Since all the functions $F, H_i$, and $G_i$ are continuous, the existence of a global solution is guaranteed by the Weierstrass Extreme Value Theorem.

\begin{theorem}[Weierstrass Extreme Value Theorem]\cite{mu00}\label{th3.2}
Every continuous function on a compact set attains its extreme values on that set.
\end{theorem}

Applying the Weierstrass Extreme Value Theorem to Problem $P_1$ guarantees the existence of a global minimum for $F(\mathbf{x})$ on the set $S_1$.

\subsection{Karush-Kuhn-Tucker Method }
In order to solve Problem $P_1$, two classical approaches in mathematical optimization will be applied, which are Karush-Kuhn-Tucker conditions and penalty function. 
The Karush-Kuhn-Tucker conditions are the first derivative tests for a solution to be optimal in nonlinear programming. Given an optimization problem, we call it Problem $P$.
\begin{problem}\label{p3.2} The optimization Problem $P$ is defined as
$$
\begin{array}{ll}
\text { $P$:Minimize } & f(\mathbf{x}) \\
\text { subject to } & g_i(\mathbf{x}) \leq 0 \text { for } i=1, \ldots, j \\
& h_i(\mathbf{x})=0 \text { for } i=1, \ldots, k \\
\end{array}
$$
where $f, g_1, \ldots, g_j, h_1, \ldots, h_{k}$ are all continuous function from $\mathbb{R}^n$ to $\mathbb{R}$.
\end{problem}
For Problem $P$, we denote the feasible regions as $S$, which is defined as
$$S=\{{\mathbf{x}}=(x_1,x_2,\ldots,x_n) \in\mathbb{R}^n|h_i({\mathbf{x}})=0,g_i({\mathbf{x}})\leq 0 \}.$$
To guarantee the set $S$ is nonempty, the Slater's condition will be checked in advance.
\begin{definition}[Slater's condition] \cite{ba06}
The Problem $P$ is said to satisfy Slater's condition if we can find a vector $\mathbf{x}\in\mathbb{R}^n$ such that
$$ g_1(\mathbf{x})<0, g_2(\mathbf{x})<0, \cdots, g_m(\mathbf{x})<0.$$
\end{definition}
The Karush-Kuhn-Tucker conditions are a set of requirements that must be satisfied by the optimal point of an optimization problem.
\begin{theorem}[Karush-Kuhn-Tucker Necessary Conditions]\cite{ba06}\label{th3.3}
Let $f: R^n \rightarrow R$, and $g_i: R^n \rightarrow R$ for $i=$ $1, \ldots, m$. Consider Problem ${P}$, which seeks to minimize $f({\mathbf{x}})$ subject to $g_i({\mathbf{x}}) \leq 0$ for $i=1, \ldots, j$ and $h_i(\mathbf{x})=0$ for $i=1,\ldots ,k$. Suppose $\overline{\mathbf{x}}$ is a solution to this problem, and that $f,g_i$ and $h_i$  are all differentiable at $\overline{\mathbf{x}}$. Furthermore, assume the gradients of the constraints, $\nabla g_i(\overline{\mathbf{x}})$, are linearly independent for all $i$. If $\overline{\mathbf{x}}$ solves Problem $\mathrm{P}$ locally, there exist scalars $u_i$ for $i \in I$ such that:
$$
\begin{aligned}
 \nabla f(\overline{\mathbf{x}})+\sum_{i=1}^{j}\lambda_i\nabla h_i(\mathbf{x})+\sum_{i=1}^k u_i \nabla g_i(\overline{\mathbf{x}}) & =\mathbf{0} & & \\
u_i g_i(\overline{\mathbf{x}}) & =0 & & \text { for } i=1, \ldots, m \\
 u_i & \geq 0 & & \text { for } i=1, \ldots, m.
\end{aligned}
$$
\end{theorem}

For the problem $P_1$ (Problem \ref{p3.1}), since we can guarantee the existence of the optimal solution $\overline{\mathbf{x}}$, this means we can find scalar $\lambda_1,\lambda_2, u_1,u_2,\ldots,u_{n-1}$ such that:
$$
\begin{aligned}
 \nabla F(\overline{\mathbf{x}})+\sum_{i=1}^{2}\lambda_i\nabla H_i(\mathbf{x})+\sum_{i=1}^{n-1} u_i \nabla G_i(\overline{\mathbf{x}}) & =\mathbf{0} & & \\
u_i G_i(\overline{\mathbf{x}}) & =0 & & \text { for } i=1, \ldots, n \\
 u_i & \geq 0 & & \text { for } i=1, \ldots, n.
\end{aligned}
$$
\subsection{Penalty Function}

The second tool used is the penalty function, which transforms a constrained optimization problem into an unconstrained one. This is done by incorporating the constraints into the objective function through a penalty parameter that penalizes any violation of the constraints.

Consider Problem $P$ again. We can replace Problem $P$ with the following unconstrained problem.
\begin{problem} The optimization Problem $P_2$ is defined as
\[
\begin{aligned}
\text{Minimize} & \quad f(\mathbf{x}) + \mu \left[\sum_{i=1}^j \left(h_i(\mathbf{x})\right)^2 + \sum_{i=1}^k \max\{g_i(\mathbf{x}), 0\}^2 \right] \\
\text{subject to} & \quad \mu \geq 0.
\end{aligned}
\]
\end{problem}
For $\mu \geq 0$, let us denote 
\[
\alpha(\mathbf{x}) = \sum_{i=1}^j \left(h_i(\mathbf{x})\right)^2 + \sum_{i=1}^k \max\{g_i(\mathbf{x}), 0\}^2,
\]
and
\[
\theta(\mu) = \inf \left\{ f(\mathbf{x}) + \mu \alpha(\mathbf{x}) \right\}.
\]
One can observe that when $\mu$ is a large positive number, any point $\mathbf{x} \in \mathbb{R}^n$ that violates the conditions in Problem $P$ will give $\alpha(\mathbf{x})$ a large value, which causes $f(\mathbf{x}) + \mu \alpha(\mathbf{x}) > \theta(\mu)$.

The following theorem concludes this subsection, which states that solving Problem $P$ is equivalent to finding the value of 
\[
\lim_{\mu \to \infty} \theta(\mu).
\]

\begin{theorem}\cite{ba06}\label{th3.4}
Consider the following problem:
\[
\begin{array}{ll}
\text{Minimize} & f(\mathbf{x}) \\
\text{subject to} & g_i(\mathbf{x}) \leq 0 \text{ for } i=1, \ldots, j \\
& h_i(\mathbf{x}) = 0 \text{ for } i=1, \ldots, k \\
\end{array}
\]
where $f, g_1, \ldots, g_j, h_1, \ldots, h_k$ are continuous functions on $\mathbb{R}^n$. Suppose that the problem has a feasible solution, and let $\alpha$ be a continuous function defined by 
\[
\alpha(\mathbf{x}) = \sum_{i=1}^j \left(h_i(\mathbf{x})\right)^2 + \sum_{i=1}^k \max\{g_i(\mathbf{x}), 0\}^2.
\]
Furthermore, suppose that for each $\mu$ there exists a solution $\mathbf{x}_\mu \in \mathbb{R}^n$ to the problem of minimizing $f(\mathbf{x}) + \mu \alpha(\mathbf{x})$, then 
\[
\inf \{ f(\mathbf{x}) \mid \mathbf{g}(\mathbf{x}) \leq \mathbf{0}, \mathbf{h}(\mathbf{x}) = \mathbf{0}, \mathbf{x} \in X \} = \lim_{\mu \to \infty} \theta(\mu),
\]
where $\theta(\mu) = \inf \{ f(\mathbf{x}) + \mu \alpha(\mathbf{x}) \mid \mathbf{x} \in \mathbb{R}^n \} = f\left(\mathbf{x}_\mu\right) + \mu \alpha\left(\mathbf{x}_\mu\right)$. Furthermore, the limit $\overline{\mathbf{x}}$ of any convergent subsequence of $\left\{\mathbf{x}_\mu\right\}$ is an optimal solution to the original problem, and $\mu \alpha\left(\mathbf{x}_\mu\right) \rightarrow 0$ as $\mu \rightarrow \infty$.
\end{theorem}

\subsection{General Procedure}

In this subsection, we describe a general procedure that combines the Karush-Kuhn-Tucker conditions and the penalty function to solve Problem $P$ when the solution exists. As the main problem of interest, Problem $P_1$ (Problem \ref{p3.1}), is guaranteed to have an optimal solution by Theorem \ref{th3.2}.

Let $\overline{\mathbf{x}}$ be the optimal solution of Problem $P$ (Problem \ref{p3.2}). According to the Karush-Kuhn-Tucker conditions, it should satisfy the following three conditions:

\begin{enumerate}
    \item Feasibility. The point $\overline{\mathbf{x}}$ lies in the feasible region, which means $h_i(\overline{\mathbf{x}}) = 0$ for $i=1,2,3,\ldots,j$ and $g_i(\overline{\mathbf{x}}) \leq 0$ for $i=1,2,3,\ldots,k$.
    \item Stationarity. The parameters $\lambda_1, \lambda_2, \ldots, \lambda_j, u_1, u_2, \ldots, u_k$ can be found such that
    \[
    \nabla f(\overline{\mathbf{x}}) + \sum_{i=1}^{j} \lambda_i \nabla h_i(\mathbf{x}) + \sum_{i=1}^k u_i \nabla g_i(\overline{\mathbf{x}}) = \mathbf{0}.
    \]
    \item Complementary Slackness. The parameters $u_1, u_2, \ldots, u_k$ should satisfy
    \begin{align*}
        u_i g_i(\overline{\mathbf{x}}) &= 0, \\
        u_i &\geq 0.
    \end{align*}
\end{enumerate}

Next, we describe the procedure to determine the exact value of $\overline{\mathbf{x}}$. First, by Condition 3, Problem $P$ can be decomposed into several subproblems. Complementary slackness stipulates that for each inequality constraint $g_i(x) \leq 0$, the corresponding Lagrange multiplier $u_i$ must satisfy $u_i g_i(x) = 0$. This means that for any optimal solution, either the constraint is active ($g_i(x) = 0$ and $u_i$ can be nonzero) or inactive ($g_i(x) < 0$ and $u_i = 0$). By enumerating all possible combinations of active and inactive constraints, $2^k$ subproblems can be generated, with each combination treated as a separate case.
For each subproblem, the active constraints are converted into equality constraints, while the inactive constraints are disregarded, with their corresponding Lagrange multipliers set to zero. Each subproblem is then solved individually. By comparing the solutions of all subproblems, the optimal solution that satisfies the original constraints and minimizes the objective function is identified.

In order to solve each subproblem, a penalty function method will be applied. For a subproblem where certain constraints $g_i(x)$ are active and others are inactive, we construct a penalized objective function:
\[
\Phi(x, \mu) = f(x) + \mu \sum_{i=1}^j h_i(x)^2 + \mu \sum_{i \in \mathcal{I}} g_i(x)^2,
\]
where $\mathcal{I} \subset \{1,2,3,\ldots,k\}$ denotes the set of active constraints, and $\mu$ is the penalty parameter. The equality constraints $h_i(x) = 0$ are always included in the penalty term. We then solve the unconstrained minimization problem:
\[
x_{\mu} = \inf_x \Phi(x, \mu)
\]
using a gradient-based optimization method. Let $x^* = \lim_{\mu \to \infty} x_\mu$, then $x^*$ will be the solution of the subproblem.

Next, we want to verify that the solution $x^*$ satisfies the stationarity condition (Condition 2) under some assumptions. By the optimality of $x_\mu$, for every $\mu \geq 0$,
\[
\nabla f(x_\mu) + \mu \sum_{i=1}^{j} 2h_i(x_\mu) \nabla h_i(x_\mu) + \mu \sum_{i \in \mathcal{I}} 2g_i(x_\mu) \nabla g_i(x_\mu) = \mathbf{0}.
\]
By letting $\mu \to \infty$,
\[
\nabla f(x^*) + \sum_{i=1}^{j} \left(\lim_{\mu \to \infty} 2\mu h_i(x_\mu)\right) \nabla h_i(x_\mu) + \sum_{i \in \mathcal{I}} \left(\lim_{\mu \to \infty} 2\mu g_i(x_\mu)\right) \nabla g_i(x_\mu) = \mathbf{0},
\]
thus if the limits in the equation exist, we can choose
\begin{align*}
    \lambda_i &= \lim_{\mu \to \infty} 2\mu h_i(x_\mu) & \text{(for $i=1,2,3,\ldots,j$)}, \\
    u_i &= \lim_{\mu \to \infty} 2\mu g_i(x_\mu) & \text{(for $i \in \mathcal{I}$)}, \\
    u_i &= 0 & \text{(for $i \notin \mathcal{I}$)},
\end{align*}
in this case, the solution $x^*$ satisfies the stationarity condition.

After obtaining solutions for each subproblem using the penalty function method, the final step is to ensure that each solution satisfies the original constraints. We check if $x^*$ complies with the feasibility (Condition 1), including equality $h_i(x) = 0$ for $i=1,2,\ldots,j$ and inequality $g_i(x) \leq 0$ for $i=1,2,\ldots,k$. If $x^*$ violates any constraint, it is considered as infeasible. Feasible solutions are then compared based on their objective function values, with the lowest value indicating the optimal solution. This process ensures that the selected solution not only meets all the constraints, but also minimizes the objective function effectively.

\vfil\pagebreak

\section{Exact value of $V_4$.}
In this chapter, we will focus on determining the exact value of $V_4$, continuing from the methodology established in the previous section. From Equation (\ref{eq1.2}), the value of $V_4$ is given by
$$
V_4=\inf _{a \in\left(1, 2\di\cos\frac{\pi}{6}\right]} \frac{\chi_4(a)}{(\sqrt{a}-1)^2},
$$

where $\chi_4(a)$ is defined as
$$
\chi_4(a)=\inf \left\{f(0)-1 \mid f(\phi)\in C_4, a_0\left(f\right)=1 \text { and } a_1\left(f\right)=a\right\}.
$$
 To find the exact value of $V_4$, we will employ the Karush-Kuhn-Tucker conditions along with the penalty function method to solve the optimization problem associated with $\chi_4(a)$. 

 Before solving the problem, it is crucial to restrict the range of $a$ to a smaller interval for computational efficiency and theoretical accuracy.  

\subsection{Revisit of Arestov and Kondrat'ev's method.}
In this subsection, we will show that
$$V_4=\inf _{\di a \in\left[1.5597515,2\cos\frac{\pi}{6}\right]} \frac{\chi_4(a)}{(\sqrt{a}-1)^2}.$$
To prove this, we revisit the approach used by Arestov and Kondratev \cite{ar90}, which provides a set of key inequalities that give $\chi_n(a)$ lower bounds. The following theorem gives a simple lower bound for $\chi_n(a)$.
\begin{theorem}\label{th4.1}\cite{ar90} For $\di a\in \left(1,2\cos\frac{\pi}{n+2}\right]$, the functions $\chi_n (a)$ for $n \geq 2$ possess the following property
\begin{align}
  \chi_n(a)\geq 2 a-1. 
\end{align}
\end{theorem}

\begin{proof} Given a cosine polynomial $f(\phi)\in C_n$ with the following expression
$$f(\phi)=1+a\cos\phi+a_2\cos 2\phi+\ldots +a_n \cos n\phi. $$
By the nonnegativity of $f(\phi)$, one has $f(\pi)\geq 0$, which implies
$$0 \leq 1-a+\sum_{k=2}^n (-1)^k a_k\leq 1-2a+\sum_{k=1}^{n} a_k,$$
this implies $\di\chi_n(a)=\sum_{k=1}^{n} a_k\geq 2a-1$.
\end{proof}

Subsequently, Arestov and Kondrat’ev employed a more sophisticated approach to establish a precise bound for $\chi_n(a)$. Given a real value function $m:\mathbb{R}\to\mathbb{R}$ such that it is nonnegative on the interval $[a,b]\subset \mathbb{R}$. Then for any $f(\phi)\in C_n$ with the following expression
$$f(\phi)=1+a\cos\phi+a_2\cos 2\phi+\ldots +a_n \cos n\phi, $$
one could define a positive functional $S:C_n\to\mathbb{R}$
$$S(f)=f(\phi_0)+\int_{a}^{b}m(\phi)f(\phi)d\phi\geq 0.$$
If we write $s(k)=S(\cos k\phi)$, then one will have 
\begin{align*}
S(f)=s(0)+as(1)+\sum_{k=2}^na_k s(k)&\geq 0.
\end{align*}
Let $\di M=\sup_{2\leq k\leq n} s(k)$, then
\begin{align*}
M\sum_{k=1}^n a_n\geq Ma+\sum_{k=2}^na_k s(k)\geq [M-s(1)] a-s(0).
\end{align*}
When $M=1$, this gives
\begin{align*}
\chi_n(a)\geq (1-s(1))a-s(0).\label{eq3}
\end{align*}
By choosing different nonnegativefunction $m(\phi)$ such that $M=1$, one could obtain several lower bounds of $\chi_n(a).$ More specifically, Arestov and Kondrat'ev \cite{ar90} have employed the following functions 
\begin{align*}
  S(f)&=f(\pi)+\int_{\pi/2}^\pi \left(3.7386644-1.4700922\cos4\phi-2.2685722\cos8\phi\right)f(\phi)d\phi,\\
  S(f)&=f\left(\frac{2\pi}{3}\right) + \int_{\pi / 3}^\pi\left(2\sqrt{3}+8\left(\phi-\frac{\pi}{3}\right)\right) f(\phi) \mathrm{d} \phi,
\end{align*}
to find another two useful lower bounds of $\chi_n(a)$.

\begin{theorem}\label{th4.2}\cite{ar90} For $\di a\in \left(1,2\cos\frac{\pi}{n+2}\right]$, the functions $\chi_n (a)$ for $ n\geq 2$ possess the following properties:
\begin{enumerate}[label=C.\arabic*, ref=a.\arabic*]
\item $\chi_n(a)\geq 5.8726781 a-6.8726781,$ 
\item $\chi_n(a)\geq 16.5 a-25.8011608.$
\end{enumerate}
\end{theorem}

These inequalities will serve as a foundation to the following theorem. The following theorem shows that the exact value of $V_4$ is the infimum of the function $\di\frac{\chi_4(a)}{(\sqrt{a}-1)^2}$ over a specific interval.

\begin{theorem}\label{th4.3}
The exact value of $V_4$ is
$$V_4=\inf _{\di a \in\left[1.5597515,2\cos\frac{\pi}{6}\right]} \frac{\chi_4(a)}{(\sqrt{a}-1)^2}.$$
\end{theorem}
\begin{proof} Since the set of cosine polynomials $C_3 \subset C_4$, one have $V_4\leq V_3=36.9199911$. However, by Theorem $\ref{th4.1}$ and Theorem $\ref{th4.2}$, we have
\begin{align}
V_4(a) &\geq F_1(a)=\frac{2a-1}{(\sqrt{a}-1)^2},\\
V_4(a) &\geq F_2(a)=\frac{5.8726781 a-6.8726781}{(\sqrt{a}-1)^2}.
\end{align}
for $a \in \left(1,2\di\cos\frac{\pi}{6}\right].$ Given function $F:(1,\infty)\to \mathbb{R}$ with the form
$$F(a)=\frac{Aa-B}{(\sqrt{a}-1)^2},$$
where $A,B$ are constant. Taking the derivative, one has
$$F'(a)=\frac{B-A \sqrt{x}}{(\sqrt{x}-1)^3 \sqrt{x}}.$$
This implies when $A>B$ the function $F(a)$ is decreasing on the interval $(1,\infty)$. On the other hand, when $A<B$, the function increasing on the interval $\di \left(1,(B/A)^2\right]$ and decreasing on the interval $\di[(B/A)^2,\infty)$. Apply this on the function $F_1(a),F_2(a)$ and $F_3(a)$, we found out $F_1(a)$ is a decreasing function on the interval $\di\left(1,2\cos\frac{\pi}{6}\right]$. However, the function $F_2(a)$ is increasing on $(1,1.3695543]$ and is decreasing on the interval $\di\left[1.3695543,2\cos\frac{\pi}{6}\right]$. 

Now, given any $f(\phi) \in C_4$ with $a_0(f)=1,a_1(f)=a\in (1,1.5275169]$, one has
$$v(f)\geq \frac{\chi_4(a)}{(\sqrt{a}-1)^2} \geq F_1(a) \geq F_1(1.5275169)=36.9199978>V_4. $$
Similarly, for any $f(\phi) \in C_4$ with $a_0(f)=1,a_1(f)=a\in [1.5275169,1.5597515]$, one could obtain
$$v(f)\geq \frac{\chi_4(a)}{(\sqrt{a}-1)^2} \geq F_2(a) \geq  F_2(1.5597515)=36.9199930>V_4. $$
This completes the proof.
\end{proof}

\subsection{Exact value of $V_4$.}

By Fejer-Reisz theorem, the problem of finding the $\chi_4(a)$ for $\di a\in\left(1, 2\di\cos\frac{\pi}{6}\right]$ can be formalized as the minimization problem in $\mathbb{R}^5$ as:
\begin{problem}\label{thm:optimization_infimum}
  Consider the optimization problem \( P_1(a) \) defined as follows:

\[
\begin{aligned}
\text{$P_{1}(a)$: Minimize } & F(\mathbf{x}) = (x_0 + x_1 + \ldots + x_4)^2 - 1 \\
\text{subject to } & H_1(\mathbf{x}) = x_0^2 + x_1^2 + x_2^2 + \ldots + x_4^2 = 1, \\
& H_2(\mathbf{x}) = 2(x_0 x_1 + x_1 x_2 + \ldots + x_{3} x_4) = a, \\
& G_1(\mathbf{x}) = 2(x_0 x_2 + x_1 x_3 + x_{2} x_4) \geq 0, \\
& G_2(\mathbf{x}) = 2(x_0 x_3 + x_1 x_4 ) \geq 0, \\
& G_{3}(\mathbf{x}) = 2x_0 x_4 \geq 0.
\end{aligned}
\]

Define \( p_{1}(a) \) as the optimal value of the objective function \( F(\mathbf{x}) \) for a given \( a \):
  
  \[
  p_{1}(a) = \min_{\mathbf{x}} F(\mathbf{x}) \quad \text{subject to the constraints of } P_1(a).
  \]
  \end{problem}

However, instead of considering Problem $P_1(a)$ with five constraints, we will consider the following problem with two constraints.

\begin{problem}\label{thm:optimization_infimum}
  Consider the optimization problem \( P_2(a) \) defined as follows:
  
  \[
  \begin{aligned}
  \text{Minimize } & F(\mathbf{x}) = (x_0 + x_1 + \ldots + x_4)^2 - 1\\
  \text{subject to } & H_1(\mathbf{x}) = \sum_{i=0}^{4} x_i^2 = 1, \\
  & H_2(\mathbf{x}) = 2\left( x_0x_1 + x_1x_2 + \ldots + x_3x_4 \right) = a.
  \end{aligned}
  \]
  
  Define \( p_2(a) \) as the optimal value of the objective function \( F(\mathbf{x}) \) for a given \( a \):
  
  \[
  p_2(a) = \min_{\mathbf{x}} F(\mathbf{x}) \quad \text{subject to the constraints of } P_2(a).
  \]
  \end{problem}

Since the feasible region of problem $P_1(a)$ is the subset of the feasible region of problem $P_2(a)$, this implies
$$\chi_4(a)=p_1(a)\geq p_2(a).$$

If the optimal point of problem $P_2(a)$ also satisfies the condition $G_1(\mathbf{x})\text{ and } G_2(\mathbf{x})$, then the optimal point will lie inside the feasible region of problem $P_1(a)$. In this case,
$$\chi_4(a)=p_1(a)=p_2(a).$$

\begin{theorem}
The exact value of \( V_4 \) is given by:

\[
\di V_4=\inf_{a \in \di \left[1.5597515, 2\cos\frac{\pi}{6}\right]} \frac{p_2(a)}{(\sqrt{a} - 1)^2} = 34.8992259.
\]

The exact solution corresponding to this infimum could be obtained at the point 
$$p=[0.2114174,0.5028452,0.6363167,0.5028451,0.2114174]\in \mathbb{R}^5.$$
It corresponds to the following polynomial in $C_4$,
$$f(\phi)=a_0+a_1\cos\phi+a_2\cos 2\phi+a_3\cos 3\phi+a_4\cos 4\phi,$$
where
\begin{align*}
a_0&=1,\\
a_1&=1.7051159,\\ 
a_2&=1.0438202,\\
a_3&=0.4252409,\\
a_4&=0.0893946.
\end{align*}
\end{theorem}

\begin{proof}
Fix an $\di a\in \left[1.5597515, 2\cos\frac{\pi}{6}\right]$, by penalty function method (Theorem \ref{th3.4}), we have
\[
p_{2}(a) = \lim_{\mu \to \infty} \theta_1(\mu),
\]
where 
\begin{align*}
\theta_1(\mu) &= \inf \left\{ F(\mathbf{x}) + \mu \left((H_1(\mathbf{x})-1)^2+(H_2(\mathbf{x})-a)^2\right) \mid \mathbf{x} \in \mathbb{R}^5 \right\}.
\end{align*}

Then, a Matlab-based algorithm is employed to compute the $\di \frac{p_{2}(a)}{(\sqrt{a}-1)^2}$. The algorithm utilizes the Newton-Raphson method in conjunction with a penalty function approach to iteratively find the minimizer of the penalized objective function. The detailed Matlab implementation is provided in \ref{Appendix B}. The Figure \ref{fig5} shows the ratio $\di \frac{p_{2}(a)}{(\sqrt{a}-1)^2}$.
\begin{figure}[H]
  \centering
  \includegraphics[width=9cm]{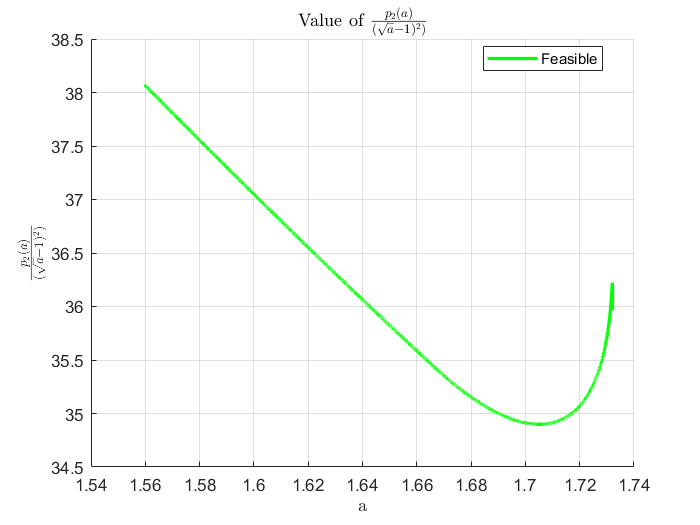}
  \caption{The ratio $\di \frac{p_{2}(a)}{(\sqrt{a}-1)^2}$.}
  \label{fig5}
\end{figure}

Then, the infimum of the ratio is determined to be \( 34.8992259 \), which could be achieved at several points, such as 
$$p=[0.2114174,0.5028451,0.6363167,0.5028451,0.2114174].$$ 
The polynomial corresponding to this solution is,
$$f(\phi)=a_0+a_1\cos\phi+a_2\cos 2\phi+a_3\cos 3\phi+a_4\cos 4\phi,$$
where
\begin{align*}
a_0&=1,\\
a_1&=1.7051159,\\ 
a_2&=1.0438202,\\
a_3&=0.4252409,\\
a_4&=0.0893946.
\end{align*}

One could see that the point $p$ also lies in the feasible region of $P_1(a_1)$. This implies that the optimal solution of problem $P_2(a_1)$ is also the optimal solution of problem $P_1(a_1)$ and thus agrees with the value of $\chi_4(a_1)$. By the relationship 
$$\frac{\chi_4(a)}{(\sqrt{a}-1)^2}\geq \frac{p_2(a)}{(\sqrt{a}-1)^2},$$
one can conclude
\begin{align*}
V_4&=\inf_{a \in \di \left[1.5597515, 2\cos\frac{\pi}{6}\right]} \frac{\chi_4(a)}{(\sqrt{a} - 1)^2}\\
&= \inf_{a \in \di \left[1.5597515, 2\cos\frac{\pi}{6}\right]} \frac{p_2(a)}{(\sqrt{a} - 1)^2}\\
&= 34.8992259.
\end{align*}
This completes the proof.
\end{proof}

\vfil\pagebreak

\section{Exact value of $V_5$.}
In this section, we aim to determine the exact value of $V_5$, following a similar method used in the previous section. From Equation (\ref{eq1.2}), the value of $V_5$ is given by

$$
V_5=\inf _{\di a \in\left(1, 2\di\cos\frac{\pi}{7}\right]} \frac{\chi_5(a)}{(\sqrt{a}-1)^2},
$$

where $\chi_5(a)$ is defined as
$$
\chi_5(a)=\inf \left\{f(0)-1 \mid f(\phi)\in C_5, a_0\left(f\right)=1 \text { and } a_1\left(f\right)=a\right\}.
$$

\subsection{Exact value of $V_5$.}

First, we will restrict the range of $a$ to a smaller interval for computational efficiency and theoretical accuracy as illustrated in the last section. The following theorem states that in order to find the exact value of $V_5$, one only needs to find the exact value of $\chi_5(a)$ for $\di a\in\left[1.6456659,2\cos\frac{\pi}{7}\right]$.

\begin{theorem}\label{th5.1}
The exact Value of $V_5$ is
$$V_5=\inf _{\di a \in\left[1.6456659,2\cos\frac{\pi}{7}\right]} \frac{\chi_5(a)}{(\sqrt{a}-1)^2}.$$
\end{theorem}
\begin{proof} Since the set of cosine polynomials $C_4 \subset C_5$, one have $V_5\leq V_4=34.8992259$. However, by Theorem $\ref{th4.1}$ and Theorem $\ref{th4.2}$, we have
\begin{align}
V_5(a) &\geq F_1(a)=\frac{2a-1}{(\sqrt{a}-1)^2},\\
V_5(a) &\geq F_2(a)=\frac{5.8726781 a-6.8726781}{(\sqrt{a}-1)^2}.
\end{align}
for $a \in \left(1,2\di\cos\frac{\pi}{7}\right].$ Given function $F:(1,\infty)\to \mathbb{R}$ with the form
$$F(a)=\frac{Aa-B}{(\sqrt{a}-1)^2},$$
where $A,B$ are constant. Taking the derivative, one have
$$F'(a)=\frac{B-A \sqrt{x}}{(\sqrt{x}-1)^3 \sqrt{x}}.$$
This implies when $A>B$ the function $F(a)$ is decreasing on the interval $(1,\infty)$. On the other hand, when $A<B$, the function increasing on the interval $\di \left(1,(B/A)^2\right]$ and decreasing on the interval $\di[(B/A)^2,\infty)$. Apply this on the function $F_1(a)$ and $F_2(a)$, we found out $F_1(a)$ is a decreasing function on the interval $\di\left(1,2\cos\frac{\pi}{7}\right]$. However, the function $F_2(a)$ is increasing on $(1,1.3695543]$ and is decreasing on the interval $\di\left[1.3695543,2\cos\frac{\pi}{7}\right]$. 

Now, given any $f(\phi) \in C_5$ with $a_0(f)=1,a_1(f)=a\in (1,1.5510971]$, one has
$$v(f)\geq \frac{\chi(a)}{(\sqrt{a}-1)^2} \geq F_1(a) \geq F_1(1.5510971)=34.8992605>V_5. $$
Similarly, for any $f(\phi) \in C_5$ with $a_0(f)=1,a_1(f)=a\in [1.5510971,1.6456659]$, one could obtain
$$v(f)\geq \frac{\chi(a)}{(\sqrt{a}-1)^2} \geq F_2(a) \geq  F_2(1.6456659)=34.8992261>V_5. $$
This completes the proof.
\end{proof}

By Fejer-Reisz theorem, the problem of finding the $\chi_5(a)$ for $\di a\in\left(1, 2\di\cos\frac{\pi}{7}\right]$ can be formalized as the minimization problem in $\mathbb{R}^6$ as:
\begin{problem}\label{thm:optimization_infimum}
  Consider the optimization problem \( Q_1(a) \) defined as follows:

\[
\begin{aligned}
\text{$Q_1(a)$: Minimize } & F(\mathbf{x}) = (x_0 + x_1 + \ldots + x_5)^2 - 1 \\
\text{subject to } & H_1(\mathbf{x}) = x_0^2 + x_1^2 + x_2^2 + \ldots + x_5^2 = 1, \\
& H_2(\mathbf{x}) = 2(x_0 x_1 + x_1 x_2 + \ldots + x_{4} x_5) = a, \\
& G_1(\mathbf{x}) = 2(x_0 x_2 + x_1 x_3 + \ldots+ x_{3} x_5) \geq 0, \\
& G_2(\mathbf{x}) = 2(x_0 x_3 + x_1 x_4 + x_{2} x_5) \geq 0, \\
& \quad\quad\quad\vdots \\
& G_{4}(\mathbf{x}) = 2x_0 x_5 \geq 0.
\end{aligned}
\]

Define \( q_1(a) \) as the optimal value of the objective function \( F(\mathbf{x}) \) for a given \( a \):
  
  \[
  q_1(a) = \min_{\mathbf{x}} F(\mathbf{x}) \quad \text{subject to the constraints of } Q_1(a).
  \]
  \end{problem}

However, instead of considering Problem $Q_1(a)$ with six constraints, we will consider the following problem with three constraints.

\begin{problem}\label{thm:optimization_infimum}
  Consider the optimization problem \( Q_2(a) \) defined as follows:
  
  \[
  \begin{aligned}
  \text{Minimize } & F(\mathbf{x}) = (x_0 + x_1 + \ldots + x_5)^2 - 1\\
  \text{subject to } & H_1(\mathbf{x}) = \sum_{i=0}^{5} x_i^2 = 1, \\
  & H_2(\mathbf{x}) = 2\left( x_0x_1 + x_1x_2 + \ldots + x_4x_5 \right) = a, \\
  & G_4(\mathbf{x}) = 2 x_0x_5  \geq 0.
  \end{aligned}
  \]
  
  Define \( q_2(a) \) as the optimal value of the objective function \( F(\mathbf{x}) \) for a given \( a \):
  
  \[
  q_2(a) = \min_{\mathbf{x}} F(\mathbf{x}) \quad \text{subject to the constraints of } Q_2(a).
  \]
  \end{problem}

Since the feasible region of problem $Q_1(a)$ is the subset of the feasible region of problem $Q_2(a)$, this implies
$$\chi_5(a)=q_1(a)\geq q_2(a).$$
If the optimal point of problem $Q_2(a)$ also satisfies the condition $G_1(\mathbf{x}), G_2(\mathbf{x}),\text{ and } G_3(\mathbf{x})$, then the optimal point will lie inside the feasible region of problem $Q_1(a)$. In this case,
$$\chi_5(a)=q_1(a)=q_2(a).$$

\begin{theorem}
The exact value of \( V_5 \) can be expressed as:

\[
V_5=\inf_{\di a \in \left[1.6456659,2\cos\frac{\pi}{7}\right]} \frac{q_2(a)}{(\sqrt{a} - 1)^2} = 34.8992259.
\]

The exact solution corresponding to this infimum could be obtained at the point
$$q=[0.2114174,0.5028451,0.6363167,0.5028451,0.2114174,0].$$ 
 The polynomial corresponding to this solution is,
$$f(\phi)=a_0+a_1\cos\phi+a_2\cos 2\phi+a_3\cos 3\phi+a_4\cos 4\phi+a_5 \cos 5\phi,$$
where
\begin{align*}
a_0&=1,\\
a_1&=1.7051159,\\ 
a_2&=1.0438202,\\
a_3&=0.4252409,\\
a_4&=0.0893946,\\
a_5&=0.
\end{align*}
\end{theorem}

\begin{proof}
Fix an $\di a\in \left[1.64566600, 2\cos\frac{\pi}{7}\right]$, we would first like to establish the numerical solution of the optimization problem \( Q_2(a) \). By applying  the Karush-Kuhn-Tucker conditions (Theorem \ref{th3.3}), if a point $\overline{\mathbf{x}}\in \mathbb{R}^6$ solves Problem $Q_2(a)$ globally, there exist scalars $\lambda_1,\lambda_2,$ and $u_1$  such that the following conditions hold:
\begin{align}
\nabla F(\overline{\mathbf{x}})+\sum_{i=1}^2 \lambda_i \nabla H_i(\overline{\mathbf{x}})+ u_1 \nabla G_{4}(\overline{\mathbf{x}}) & =\mathbf{0}, & \\
u_1 G_{4}(\overline{\mathbf{x}}) & =0, \label{eq8.5} \\
u_1 & \geq 0. 
\end{align}
By the Equation (\ref{eq8.5}), the problem $Q_2(a)$ could be divided into the following two subproblems, we call them problem $Q_{2,1}(a)$ and $Q_{2,2}(a)$ respectively. We will use $q_{2,1}(a)$ and $q_{2,2}(a)$ to denote the solutions of these subproblems respectively. The first subcase is when $u_1=0$, the minimization problem corresponding to this subcase is

\[
\begin{aligned}
\text{$Q_{2,1}(a)$: Minimize } & F(\mathbf{x}) = (x_0 + x_1 + \ldots + x_5)^2 - 1 \\
\text{subject to } & H_1(\mathbf{x}) = x_0^2 + x_1^2 + x_2^2 + \ldots + x_5^2 = 1, \\
& H_2(\mathbf{x}) = 2(x_0 x_1 + x_1 x_2 + \ldots + x_{4} x_5) = a. 
\end{aligned}
\]

The second subcase is when $u_1\neq 0$, the minimization problem corresponding to this subcase is    

\[ 
\begin{aligned} 
\text{$Q_{2,2}(a)$: Minimize } & F(\mathbf{x}) = (x_0 + x_1 + \ldots + x_5)^2 - 1 \\
\text{subject to } & H_1(\mathbf{x}) = x_0^2 + x_1^2 + x_2^2 + \ldots + x_5^2 = 1, \\
& H_2(\mathbf{x}) = 2(x_0x_1 + x_1 x_2 + \ldots + x_{4} x_5) = a, \\
& G_4(\mathbf{x}) = 2x_0 x_5=0.
\end{aligned}
\]

By penalty function method (Theorem \ref{th3.4}), for $i=1,2$, we have
\[
q_{2,i}(a) = \lim_{\mu \to \infty} \theta_i(\mu),
\]
where 
\begin{align*}
\theta_1(\mu) &= \inf \left\{ F(\mathbf{x}) + \mu \left((H_1(\mathbf{x})-1)^2+(H_2(\mathbf{x})-a)^2\right) \mid \mathbf{x} \in \mathbb{R}^6 \right\}, \\
\theta_2(\mu) &= \inf \left\{ F(\mathbf{x}) + \mu \left((H_1(\mathbf{x})-1)^2+(H_2(\mathbf{x})-a)^2+G_4(\mathbf{x})^2\right) \mid \mathbf{x} \in \mathbb{R}^6 \right\}.
\end{align*}

Then, a Matlab-based algorithm is employed to compute the $\di \frac{q_{2,i}(a)}{(\sqrt{a}-1)^2}$ for $i=1,2$. The algorithm utilizes the Newton-Raphson method in conjunction with a penalty function approach to iteratively find the minimizer of the penalized objective function. The detailed Matlab implementation is provided in \ref{Appendix B}.

For each subproblem \( Q_{2,i}(a) \), a graph (see Figure \ref{fig6} and Figure \ref{fig7}) is generated plotting the ratio \(\di \frac{p_{2,i}(a)}{(\sqrt{a} - 1)^2} \) against \( a \). In these graphs, we use different colors to represent the feasible and infeasible solutions to the original problem $Q_2(a)$. The red points are used to indicate infeasible solutions that are the solutions of subproblems but do not satisfy all the constraints of problem $Q_2(a)$. On the other hand, the green points represent feasible solutions that satisfy all constraints.
\begin{figure}[H]
  \centering
  \includegraphics[width=9cm]{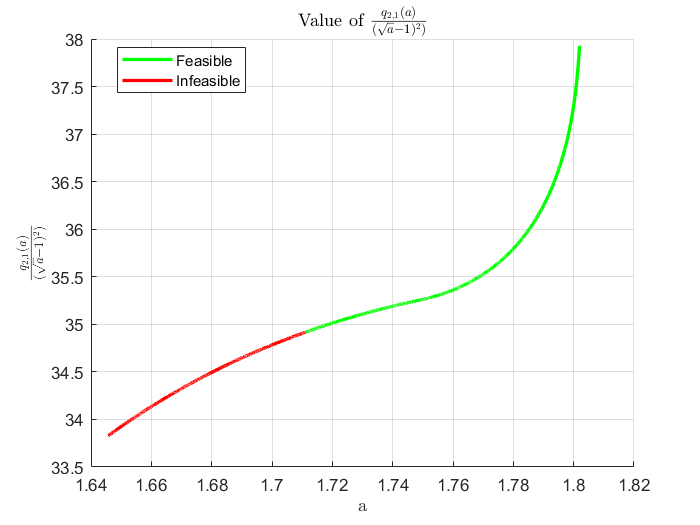}
  \caption{The ratio $\di \frac{q_{2,1}(a)}{(\sqrt{a}-1)^2}$.}
  \label{fig6}
  \end{figure}

  \begin{figure}[H]
    \centering
    \includegraphics[width=9cm]{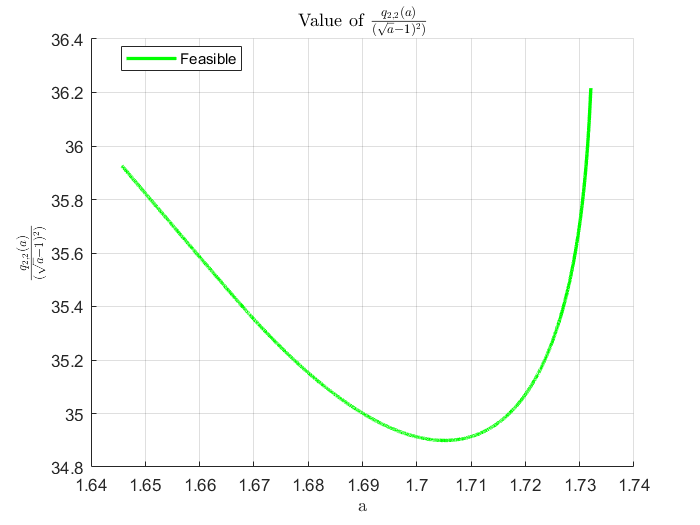}
    \caption{The ratio $\di \frac{q_{2,2}(a)}{(\sqrt{a}-1)^2}$.}
    \label{fig7}
  \end{figure}

Then, the value of $q_2(a)$ could be obtained by comparing the optimal values of the subproblems $q_{2,1}(a)$ and $q_{2,2}(a)$ (only consider the value feasible to the problem $Q_2(a)$). The Figure \ref{fig8} shows the ratio $\di \frac{q_{2}(a)}{(\sqrt{a}-1)^2}$. 

\begin{figure}[H]
  \centering
  \includegraphics[width=9cm]{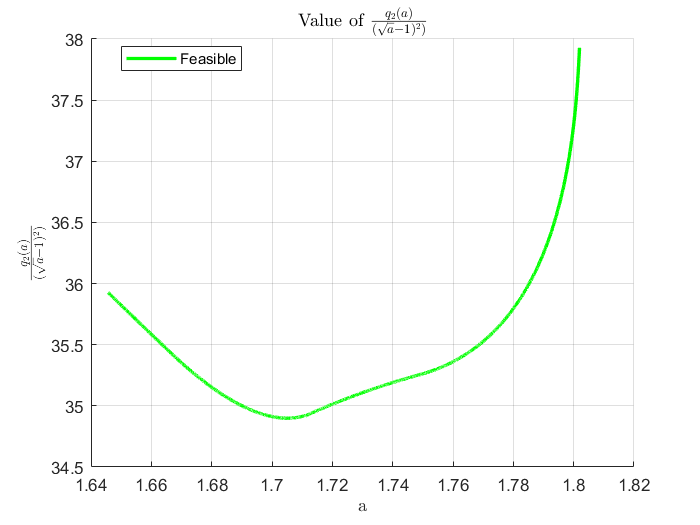}
  \caption{The ratio $\di \frac{q_{2}(a)}{(\sqrt{a}-1)^2}$.}
  \label{fig8}
  \end{figure}

  Then, the infimum of the ratio is determined to be \( 34.8992259 \), which could be achieved at several points, such as 
$$q=[0.2114174,0.5028451,0.6363167,0.5028452,0.2114174,0].$$ 
The polynomial corresponding to this solution is,
$$f(\phi)=a_0+a_1\cos\phi+a_2\cos 2\phi+a_3\cos 3\phi+a_4\cos 4\phi+a_5 \cos 5\phi,$$
where
\begin{align*}
a_0&=1,\\
a_1&=1.7051159,\\ 
a_2&=1.0438202,\\
a_3&=0.4252409,\\
a_4&=0.0893946,\\
a_5&=0.
\end{align*}

One could see that this point $q$ also lies in the feasible region of $Q_1(a_1)$. This implies that the optimal solution of problem $Q_2(a_1)$ is also the optimal solution of problem $Q_1(a_1)$ and thus agrees with the value of $\chi_5(a_1)$. By the relationship 
$$\frac{\chi_5(a)}{(\sqrt{a}-1)^2}\geq \frac{q_2(a)}{(\sqrt{a}-1)^2},$$
one can conclude
\begin{align*}
V_5&=\inf_{a \in \di \left[1.5597515, 2\cos\frac{\pi}{7}\right]} \frac{\chi_5(a)}{(\sqrt{a} - 1)^2}\\
&= \inf_{a \in \di \left[1.5597515, 2\cos\frac{\pi}{7}\right]} \frac{q_2(a)}{(\sqrt{a} - 1)^2}\\
&= 34.8992259.
\end{align*}
This completes the proof.

\end{proof}

\vfil\pagebreak

\section{Exact value of $V_6$.}
In this section, we will determine the exact value of $V_6$. From Equation (\ref{eq1.2}), we have
$$
V_6=\inf _{a \in\left(1, 2\di\cos\frac{\pi}{8}\right]} \frac{\chi_6(a)}{(\sqrt{a}-1)^2},
$$

where $\chi_6(a)$ is defined as
$$
\chi_6(a)=\inf \left\{f(0)-1 \mid f(\phi)\in C_6, a_0\left(f\right)=1 \text { and } a_1\left(f\right)=a\right\}.
$$

\subsection{Exact value of $V_6$.}
In order to find the exact value of $V_6$, we will use the Karush-Kuhn-Tucker conditions and the penalty function method to solve the optimization problem $\chi_6(a)$. However, before proceeding, we first restrict the range of a to a smaller interval.

By a similar proof in Theorem \ref{th5.1}, we have the following theorem.

\begin{theorem}\label{th6.1} The exact value of $V_6$ could be obtained by
$$\di V_6=\inf_{a\in[1.6456659,1.8231801]} \frac{\chi_6(a)}{(\sqrt{a}-1)^2}.$$
\end{theorem}
\begin{proof}
  Since the set of cosine polynomials $C_5 \subset C_6$, one have $V_6 \leq V_5=$ 34.8992259. However, by Theorem 4.1 and Theorem 4.2 , we have

  \begin{align*}
   &V_6(a) \geq F_1(a)=\frac{2 a-1}{(\sqrt{a}-1)^2} \\
   &V_6(a) \geq F_2(a)=\frac{5.87267810 a-6.87267810}{(\sqrt{a}-1)^2} \\
  &V_6(a) \geq F_3(a)=\frac{16.5 a-25.8011608}{(\sqrt{a}-1)^2}
  \end{align*}

  for $\di a \in\left(1,2 \cos \frac{\pi}{8}\right]$. Given function $F:(1, \infty) \rightarrow \mathbb{R}$ with the form
  
  $$
  F(a)=\frac{A a-B}{(\sqrt{a}-1)^2}
  $$
  
  where $A, B$ are constant. Taking the derivative, one have
  
  $$
  F^{\prime}(a)=\frac{B-A \sqrt{x}}{(\sqrt{x}-1)^3 \sqrt{x}}
  $$
  This implies when $A>B$ the function $F(a)$ is decreasing on the interval $(1, \infty)$. On the other hand, when $A<B$, the function increasing on the interval $\left(1,(B / A)^2\right]$ and decreasing on the interval $\left[(B / A)^2, \infty\right)$. Apply this on the function $F_1(a), F_2(a)$ and $F_3(a)$, we found out $F_1(a)$ is a decreasing function on the interval $\di\left(1,2 \cos \frac{\pi}{8}\right]$. However, the function $F_2(a)$ is increasing on $(1,1.3695543]$ and is decreasing on the interval $\di \left[1.3695543,2 \cos \frac{\pi}{8}\right]$. At the same time, the function $F_3(a)$ is increasing on the interval $\di\left(1,2 \cos \frac{\pi}{8}\right]$.
Now, given any $f(\phi) \in C_6$ with $a_0(f)=1, a_1(f)=a \in(1,1.5510971]$, one has

$$
v(f) \geq \frac{\chi(a)}{(\sqrt{a}-1)^2} \geq F_1(a) \geq F_1(1.5510971)=34.8992605>V_5 .
$$

Similarly, for any $f(\phi) \in C_5$ with $a_0(f)=1, a_1(f)=a \in[1.5510971,1.6456659]$, one could obtain

$$
v(f) \geq \frac{\chi(a)}{(\sqrt{a}-1)^2} \geq F_2(a) \geq F_2(1.6456659)=34.8992261>V_5
$$

Finally, for any $\di f(\phi) \in C_6$ with $\di a_0(f)=1, a_1(f)=a \in\left[1.8231801,2 \cos \frac{\pi}{8}\right]$, one could obtain

$$
v(f) \geq \frac{\chi(a)}{(\sqrt{a}-1)^2} \geq F_3(a) \geq F_3(1.8231801)=34.8992630>V_6
$$

This completes the proof.

\end{proof}

By Fejer-Reisz theorem theorem, the problem of finding the $\chi_6(a)$ for $\di a\in\left(1, 2\di\cos\frac{\pi}{8}\right]$ can be formalized as the minimization problem in $\mathbb{R}^7$ as:
\begin{problem}\label{thm:optimization_infimum}
  Consider the optimization problem \( R_1(a) \) defined as follows:

\[
\begin{aligned}
\text{$R_1(a)$: Minimize } & F(\mathbf{x}) = (x_0 + x_1 + \ldots + x_6)^2 - 1 \\
\text{subject to } & H_1(\mathbf{x}) = x_0^2 + x_1^2 + x_2^2 + \ldots + x_6^2 = 1, \\
& H_2(\mathbf{x}) = 2(x_0 x_1 + x_1 x_2 + \ldots + x_{5} x_6) = a, \\
& G_1(\mathbf{x}) = 2(x_0 x_2 + x_1 x_3 + \ldots+ x_{4} x_6) \geq 0, \\
& G_2(\mathbf{x}) = 2(x_0 x_3 + x_1 x_4 + \ldots+ x_{3} x_6) \geq 0, \\
& \quad\quad\quad\vdots \\
& G_{5}(\mathbf{x}) = 2x_0 x_6 \geq 0.
\end{aligned}
\]

Define \( r_1(a) \) as the optimal value of the objective function \( F(\mathbf{x}) \) for a given \( a \):
  
  \[
  r_1(a) = \min_{\mathbf{x}} F(\mathbf{x}) \quad \text{subject to the constraints of } R_1(a).
  \]
  \end{problem}

However, instead of consider Problem $R_1(a)$ with seven constraints, we will consider following problem with four constraints.

\begin{problem}\label{thm:optimization_infimum}
  Consider the optimization problem \( R_2(a) \) defined as follows:
  
  \[
  \begin{aligned}
  \text{Minimize } & F(\mathbf{x}) = (x_0 + x_1 + \ldots + x_6)^2 - 1 \\
  \text{subject to } & H_1(\mathbf{x}) = \sum_{i=0}^{6} x_i^2 = 1, \\
  & H_2(\mathbf{x}) = 2\left( x_0x_1 + x_1x_2 + \ldots + x_5x_6 \right) = a, \\
  & G_4(\mathbf{x}) = 2\left( x_0x_5 + x_1x_6  \right) \geq 0, \\
  & G_5(\mathbf{x}) = 2 x_0x_6  \geq 0.
  \end{aligned}
  \]
  
  Define \( r_2(a) \) as the optimal value of the objective function \( F(\mathbf{x}) \) for a given \( a \):
  
  \[
  r_2(a) = \min_{\mathbf{x}} F(\mathbf{x}) \quad \text{subject to the constraints of } R_2(a).
  \]
  \end{problem}

Since the feasible region of problem $R_1(a)$ is the subset of the feasible region of problem $R_2(a)$, this implies
$$\chi_6(a)=r_1(a)\geq r_2(a).$$
If the optimal point of problem $R_2(a)$ also satisfies the condition $G_1(\mathbf{x}), G_2(\mathbf{x})\text{ and } G_3(\mathbf{x})$, then the optimal point will lie inside the feasible region of problem $R_1(a)$. In this case,
$$\chi_6(a)=r_1(a)=r_2(a).$$

\begin{theorem}
The exact value of \( V_6 \) can be expressed as:

\[
\di V_6=\inf_{a \in [1.6456659,1.8231801]} \frac{r_2(a)}{(\sqrt{a} - 1)^2} = 34.8992259.
\]

The exact solution corresponding to this infimum could be obtained at the point
$$r=[0.2114174,0.5028451,0.6363167,0.5028452,0.2114174,0,0].$$ 
 The polynomial corresponding to this solution is,
$$f(\phi)=a_0+a_1\cos\phi+a_2\cos 2\phi+a_3\cos 3\phi+a_4\cos 4\phi+a_5 \cos 5\phi+a_6\cos 6\phi,$$
where
\begin{align*}
a_0&=1,\\
a_1&=1.7051159,\\ 
a_2&=1.0438202,\\
a_3&=0.4252409,\\
a_4&=0.0893946,\\
a_5&=0,\\
a_6&=0.
\end{align*}
\end{theorem}

\begin{proof}
Fix an $a\in [1.64566600, 1.82318005]$ ,we would first like to establish the numerical solution of the optimization problem \( R_2(a) \). By applying  the Karush-Kuhn-Tucker conditions (Theorem \ref{th3.3}), if a point $\overline{\mathbf{x}}\in \mathbb{R}^7$ solves Problem $R_2(a)$ globally, there exist scalars $\lambda_1,\lambda_2,u_1,$ and $u_2$  such that

\begin{align}
\nabla F(\overline{\mathbf{x}})+\sum_{i=1}^2 \lambda_i \nabla H_i(\overline{\mathbf{x}})+\sum_{i=1}^2 u_i \nabla G_{3+i}(\overline{\mathbf{x}}) & =\mathbf{0} & \\
u_i G_{3+i}(\overline{\mathbf{x}}) & =0 & \text { for } i=1, 2\label{eq6.4} \\
u_i & \geq 0 & \text { for } i=1, 2
\end{align}

By the Equation (\ref{eq6.4}), the problem $R_2(a)$ could be divided into following four subproblems, we call them problem $R_{2,1}(a),R_{2,2}(a),R_{2,3}(a)$ and $R_{2,4}(a)$ respectively. We will use $r_{2,1}(a),r_{2,2}(a),r_{2,3}(a)$ and $r_{2,4}(a)$ to denote the solutions of these subproblems respecively. The first subcase is when $u_1=0, u_2=0$, the minimization problem corresponding to this subcase is

\[
\begin{aligned}
\text{$R_{2,1}(a)$: Minimize } & F(\mathbf{x}) = (x_0 + x_1 + \ldots + x_6)^2 - 1 \\
\text{subject to } & H_1(\mathbf{x}) = x_0^2 + x_1^2 + x_2^2 + \ldots + x_6^2 = 1, \\
& H_2(\mathbf{x}) = 2(x_0 x_1 + x_1 x_2 + \ldots + x_{5} x_6) = a. 
\end{aligned}
\]

The second subcase is when $u_1\neq 0, u_2=0$, the minimization problem corresponding to this subcase is    

\[ 
\begin{aligned} 
\text{$R_{2,2}(a)$: Minimize } & F(\mathbf{x}) = (x_0 + x_1 + \ldots + x_6)^2 - 1 \\
\text{subject to } & H_1(\mathbf{x}) = x_0^2 + x_1^2 + x_2^2 + \ldots + x_6^2 = 1, \\
& H_2(\mathbf{x}) = 2(x_0 x_1 + x_1 x_2 + \ldots + x_{5} x_6) = a, \\
& G_4(\mathbf{x}) = 2(x_0 x_5 + x_1 x_6)=0.
\end{aligned}
\]

The third subcase is when $u_1=0, u_2\neq 0$, the minimization problem corresponding to this subcase is

\[
\begin{aligned}
\text{$R_{2,3}(a)$: Minimize } & F(\mathbf{x}) = (x_0 + x_1 + \ldots + x_6)^2 - 1 \\
\text{subject to } & H_1(\mathbf{x}) = x_0^2 + x_1^2 + x_2^2 + \ldots + x_6^2 = 1, \\
& H_2(\mathbf{x}) = 2(x_0 x_1 + x_1 x_2 + \ldots + x_{5} x_6) = a, \\
& G_5(\mathbf{x}) = 2x_0 x_6 =0.
\end{aligned}
\]

The fourth subcase is when $u_1\neq 0, u_2\neq 0$, the minimization problem corresponding to this subcase is

\[
\begin{aligned}
\text{$R_{2,4}(a)$: Minimize } & F(\mathbf{x}) = (x_0 + x_1 + \ldots + x_6)^2 - 1 \\
\text{subject to } & H_1(\mathbf{x}) = x_0^2 + x_1^2 + x_2^2 + \ldots + x_6^2 = 1, \\
& H_2(\mathbf{x}) = 2(x_0 x_1 + x_1 x_2 + \ldots + x_{5} x_6) = a, \\
& G_4(\mathbf{x}) = 2(x_0 x_5 + x_1 x_6)=0, \\
& G_5(\mathbf{x}) = 2x_0 x_6 =0.
\end{aligned}
\]

By penalty function method (Theorem \ref{th3.4}), for $i=1,2,3,4$, we have
\[
r_{2,i}(a) = \lim_{\mu \to \infty} \theta_i(\mu),
\]
where 
\begin{align*}
\theta_1(\mu) &= \inf \left\{ F(\mathbf{x}) + \mu \left((H_1(\mathbf{x})-1)^2+(H_2(\mathbf{x})-a)^2\right) \mid \mathbf{x} \in \mathbb{R}^7 \right\}, \\
\theta_2(\mu) &= \inf \left\{ F(\mathbf{x}) + \mu \left((H_1(\mathbf{x})-1)^2+(H_2(\mathbf{x})-a)^2+G_4(\mathbf{x})^2\right) \mid \mathbf{x} \in \mathbb{R}^7 \right\},\\
\theta_3(\mu) &= \inf \left\{ F(\mathbf{x}) + \mu \left((H_1(\mathbf{x})-1)^2+(H_2(\mathbf{x})-a)^2+G_5(\mathbf{x})^2\right) \mid \mathbf{x} \in \mathbb{R}^7 \right\},\\
\theta_4(\mu) &= \inf \left\{ F(\mathbf{x}) + \mu \left((H_1(\mathbf{x})-1)^2+(H_2(\mathbf{x})-a)^2+\sum_{i=1}^2 G_{3+i}(\mathbf{x})^2\right) \mid \mathbf{x} \in \mathbb{R}^7 \right\}.
\end{align*}

Then, a Matlab-based algorithm is employed to compute the $\di \frac{r_{2,i}(a)}{(\sqrt{a}-1)^2}$ for $i=1,2,3,4$. The algorithm employs the Newton-Raphson method, combined with a penalty function approach to iteratively minimize the penalized objective function (see \ref{Appendix B} for the full Matlab implementation).

For each subproblem \( R_{2,i}(a) \), a graph (see Figure\ref{fig9}, Figure\ref{fig10}, Figure\ref{fig11} and Figure\ref{fig12}) is generated plotting the ratio \(\di \frac{r_{2,i}(a)}{(\sqrt{a} - 1)^2} \) against \( a \). In these graphs, different colors are used to distinguish between feasible and infeasible solutions to the original problem $R_2(a)$. Red points indicate infeasible solutions, which solve subproblems but fail to meet all the constraints, while green points represent feasible solutions that satisfy all constraints.
\begin{figure}[H]
  \centering
  \includegraphics[width=9cm]{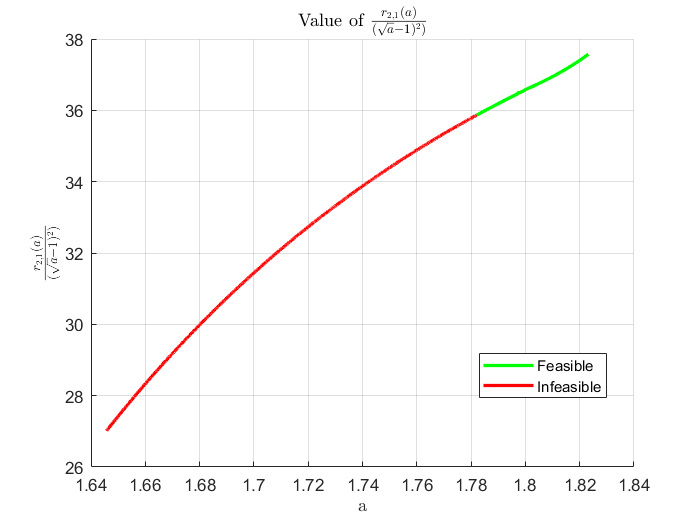}
  \caption{The ratio $\di \frac{r_{2,1}(a)}{(\sqrt{a}-1)^2}$.}
  \label{fig9}
  \end{figure}

  \begin{figure}[H]
    \centering
    \includegraphics[width=9cm]{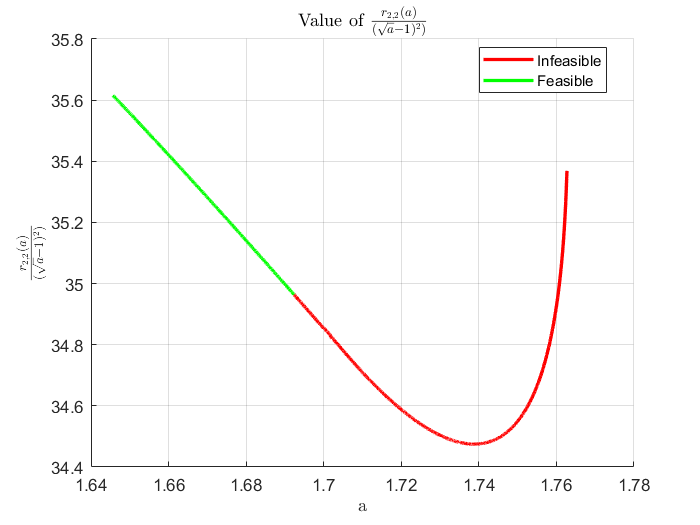}
    \caption{The ratio $\di \frac{r_{2,2}(a)}{(\sqrt{a}-1)^2}$.}
    \label{fig10}
  \end{figure}

    \begin{figure}[H]
      \centering
      \includegraphics[width=9cm]{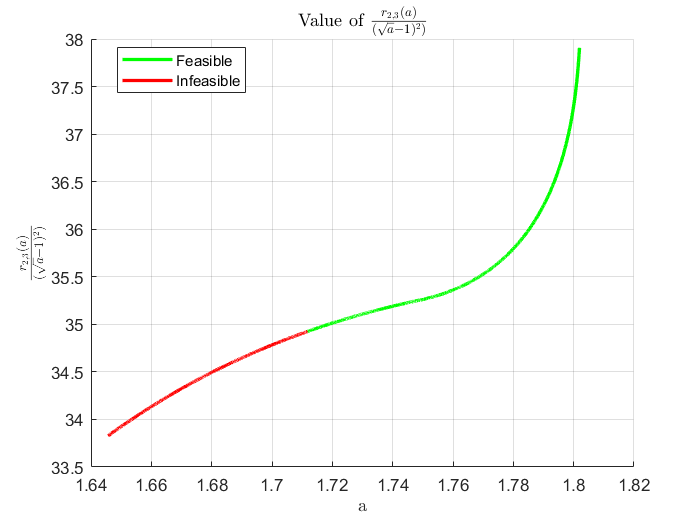}
      \caption{The ratio $\di \frac{r_{2,3}(a)}{(\sqrt{a}-1)^2}$.}
      \label{fig11}
      \end{figure}

\begin{figure}[H]
        \centering
        \includegraphics[width=9cm]{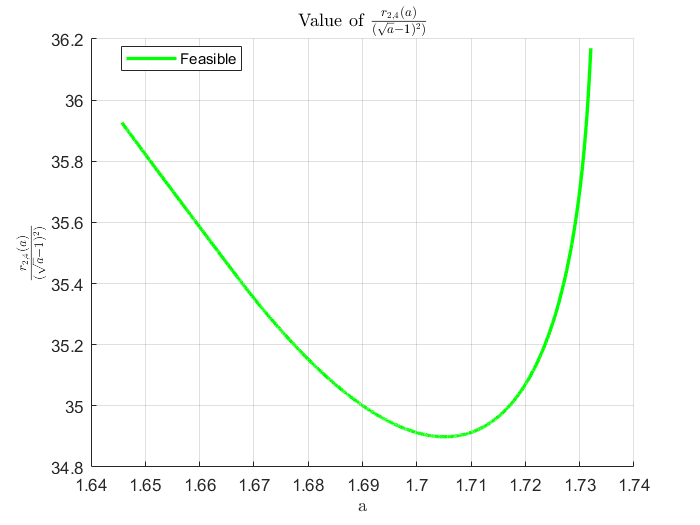}
        \caption{ The ratio $\di \frac{r_{2,4}(a)}{(\sqrt{a}-1)^2}$.}
        \label{fig12}
        \end{figure}
      
Among all feasible solutions for each \( a \), \( r_2(a) \) is identified as the minimum value of all $r_{2,i}(a)$ that are feasible to the problem $R_2(a)$. Figure \ref{fig13} shows the graph plotting the ratio of the function $\frac{r_{2}(a)}{(\sqrt{a}-1)^2}$.

\begin{figure}[H]
  \centering
  \includegraphics[width=9cm]{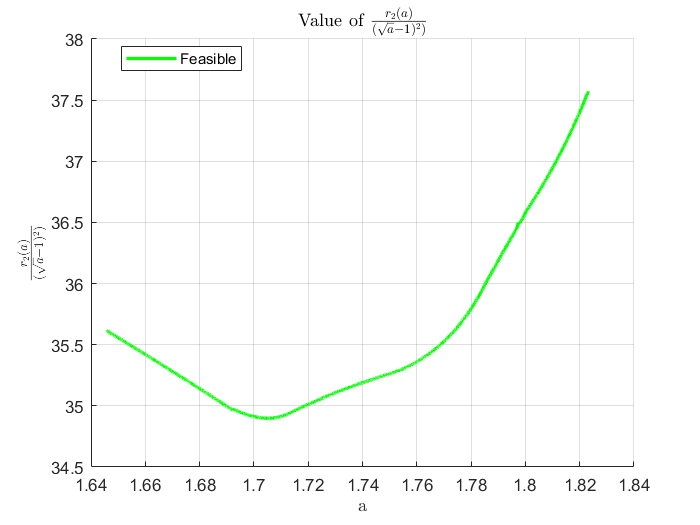}
  \caption{The ratio $\di \frac{r_{2}(a)}{(\sqrt{a}-1)^2}$.}
  \label{fig13}
  \end{figure}
  
  Then, the infimum of the ratio is determined to be \( 34.899226 \), achieved at a specific point in \( \mathbb{R}^7 \). This point is given by 
$$r=[0.2114174,0.5028451,0.6363167,0.5028452,0.2114174,0,0].$$ 
 The polynomial corresponding to this point is,
$$f(\phi)=a_0+a_1\cos\phi+a_2\cos 2\phi+a_3\cos 3\phi+a_4\cos 4\phi+a_5 \cos 5\phi+a_6\cos 6\phi,$$
where
\begin{align*}
a_0&=1,\\
a_1&=1.7051159,\\ 
a_2&=1.0438202,\\
a_3&=0.4252409,\\
a_4&=0.0893946,\\
a_5&=0,\\
a_6&=0.
\end{align*}

One could see that this point $r$ also lies in the feasible region of $R_1(a_1)$. This implies that the optimal solution of problem $R_2(a_1)$ is also the optimal solution of problem $R_1(a_1)$ and thus agrees with the value of $\chi_6(a_1)$. By the relationship 
$$\frac{\chi_6(a)}{(\sqrt{a}-1)^2}\geq \frac{r_2(a)}{(\sqrt{a}-1)^2},$$
one can conclude
\begin{align*}
V_6&=\inf_{a \in \di [1.6456659,1.8231801]} \frac{\chi_6(a)}{(\sqrt{a} - 1)^2}\\
&= \inf_{a \in \di [1.6456659,1.8231801]} \frac{r_2(a)}{(\sqrt{a} - 1)^2}\\
&= 34.8992259.
\end{align*}
This completes the proof.

\end{proof}
\vfil\pagebreak

\section{Exact value of $V_7$.}
This section will be devoted to find the exact value of $V_7$. First of all, by Equation (\ref{eq1.2}), we have
$$
\di V_7=\inf _{a \in\left(1, 2\di\cos\frac{\pi}{9}\right]} \frac{\chi_7(a)}{(\sqrt{a}-1)^2},
$$
where
$$\chi_7(a)=\inf \left\{f(0)-1 \mid f(\phi)\in C_7, a_0\left(f\right)=1 \text { and } a_1\left(f\right)=a\right\} \text {. }$$

\subsection{Exact value of $V_7$.}

In order to find the exact value of $V_7$, we will use the Karush-Kuhn-Tucker conditions and the penalty function method to solve the optimization problem $\chi_7(a)$. Before proceeding, we first restrict the range of a to a smaller interval.

By the same proof in Theorem \ref{th6.1}, we have the following theorem.

\begin{theorem}\label{th7.1} The exact value of $V_7$ could be obtained by
  $$\di V_7=\inf_{a \in a\in[1.6456659,1.8231801]} \frac{\chi_7(a)}{(\sqrt{a}-1)^2}.$$
  \end{theorem}

By Fejer-Reisz theorem theorem, the problem of finding the $\chi_7(a)$ for $\di a\in\left(1, 2\di\cos\frac{\pi}{9}\right]$ can be formalised as the following minimization problem in $\mathbb{R}^8$. 

\begin{problem}  Consider the optimization problem \( S_1(a) \) defined as follows:
\[
\begin{aligned}
\text{$S_1(a)$: Minimize } & F(\mathbf{x}) = (x_0 + x_1 + \ldots + x_7)^2 - 1 \\
\text{subject to } & H_1(\mathbf{x}) = x_0^2 + x_1^2 + x_2^2 + \ldots + x_7^2 = 1, \\
& H_2(\mathbf{x}) = 2(x_0 x_1 + x_1 x_2 + \ldots + x_{6} x_7) = a, \\
& G_1(\mathbf{x}) = 2(x_0 x_2 + x_1 x_3 + \ldots+ x_{5} x_7) \geq 0, \\
& G_2(\mathbf{x}) = 2(x_0 x_3 + x_1 x_4 + \ldots+ x_{4} x_7) \geq 0, \\
& \quad\quad\quad\vdots \\
& G_{6}(\mathbf{x}) = 2x_0 x_7 \geq 0.
\end{aligned}
\]

Define $s_1(a)$ as the optimal value of the objective function $F(\mathbf{x})$ for a given $a$:

\[
s_1(a) = \min_{\mathbf{x}} F(\mathbf{x}) \quad \text{subject to the constraints of } S_1(a).
\]
\end{problem}

However, instead of consider Problem $S_1(a)$ with eight constraints, we will consider following problem with four constraints.

\begin{problem}\label{thm:optimization_infimum}
Consider the optimization problem \( S_2(a) \) defined as follows:

\[
\begin{aligned}
\text{$S_2(a)$: Minimize } & F(\mathbf{x}) = (x_0 + x_1 + \ldots + x_7)^2 - 1 \\
\text{subject to } & H_1(\mathbf{x}) = x_0^2 + x_1^2 + x_2^2 + \ldots + x_7^2 = 1, \\
& H_2(\mathbf{x}) = 2(x_0 x_1 + x_1 x_2 + \ldots + x_{6} x_7) = a, \\
& G_4(\mathbf{x}) = 2(x_0 x_5 + x_1 x_6+x_2x_7 ) \geq 0, \\
& G_5(\mathbf{x}) = 2(x_0 x_6 + x_1 x_7) \geq 0.
\end{aligned}
\]

Define \( s_2(a) \) as the optimal value of the objective function \( F(\mathbf{x}) \) for a given \( a \):

\[
s_2(a) = \min_{\mathbf{x}} F(\mathbf{x}) \quad \text{subject to the constraints of } S_2(a).
\]
\end{problem}

Since the feasible region of problem $S_1(a)$ is the subset of the feasible region of problem $S_2(a)$, this implies
$$\chi_7(a)=s_1(a)\geq s_2(a).$$
If the optimal point of problem $S_2(a)$ also satisfies the condition $G_1(\mathbf{x}), G_2(\mathbf{x}), G_3(\mathbf{x}),\text{ and } G_6(\mathbf{x})$, then the optimal point will lie inside the feasible region of problem $S_1(a)$. In this case,
$$\chi_7(a)=s_1(a)=s_2(a).$$

\begin{theorem}
The exact value of \( V_7 \) can be expressed as:

$$\di V_7=\inf_{a \in a\in[1.6456659,1.8231801]} \frac{\chi_7(a)}{(\sqrt{a}-1)^2}=34.6494874.$$
The exact solution corresponding to this infimum could be obtained at the point
\begin{align*}
s&=[0.1685903,0.4506317,0.6267577,0.5454191,\\
&0.2756348,0.0361770,-0.0282171,0.1055656].
\end{align*}
 The polynomial corresponding to this solution is,
\begin{align*}
 f(\phi)=&a_0+a_1\cos\phi+a_2\cos 2\phi+a_3\cos 3\phi+a_4\cos 4\phi\\ &+a_5 \cos 5\phi+a_6\cos 6\phi+a_7\cos 7\phi,
\end{align*}
 where
\begin{align*}
a_0&=1,\\
a_1&=1.7185098,\\ 
a_2&=1.0731034,\\
a_3&=0.4527292,\\
a_4&=0.1016950,\\
a_5&=0,\\
a_6&=0,\\
a_7&=0.0035595.
\end{align*}

\end{theorem}

\begin{proof}
Fix an $a\in [1.6456659, 1.8231801]$, we would first like to establish the numerical solution of the optimization problem \( S_2(a) \). By applying  the Karush-Kuhn-Tucker conditions (Theorem \ref{th3.3}), if  a point $\overline{\mathbf{x}}\in \mathbb{R}^8$ solves Problem $S_2(a)$ globally, there exist scalars $\lambda_1,\lambda_2,u_1,$ and $u_2$  such that

\begin{align}
\nabla F(\overline{\mathbf{x}})+\sum_{i=1}^2 \lambda_i \nabla H_i(\overline{\mathbf{x}})+\sum_{i=1}^2 u_i \nabla G_{3+i}(\overline{\mathbf{x}}) & =\mathbf{0} & \\
u_i G_{3+i}(\overline{\mathbf{x}}) & =0 & \text { for } i=1, 2\label{eq7.2} \\
u_i & \geq 0 & \text { for } i=1, 2
\end{align}

By the Equation (\ref{eq7.2}), the problem $S_2(a)$ could be divided into the following four subproblems, we call them problem $S_{2,1}(a), S_{2,2}(a), S_{2,3}(a)$ and $S_{2,4}(a)$ respectively. We will use $s_{2,1}(a),s_{2,2}(a),s_{2,3}(a)$ and $s_{2,4}(a)$ to denote the solutions of these subproblems respectively. The first subcase is when $u_1=0, u_2=0$, the minimization problem corresponding to this subcase is

\[
\begin{aligned}
\text{$S_{2,1}(a)$: Minimize } & F(\mathbf{x}) = (x_0 + x_1 + \ldots + x_7)^2 - 1 \\
\text{subject to } & H_1(\mathbf{x}) = x_0^2 + x_1^2 + x_2^2 + \ldots + x_7^2 = 1, \\
& H_2(\mathbf{x}) = 2(x_0 x_1 + x_1 x_2 + \ldots + x_{6} x_7) = a. 
\end{aligned}
\]

The second subcase is when $u_1\neq 0, u_2=0$, the minimization problem corresponding to this subcase is    

\[ 
\begin{aligned} 
\text{$S_{2,2}(a)$: Minimize } & F(\mathbf{x}) = (x_0 + x_1 + \ldots + x_7)^2 - 1 \\
\text{subject to } & H_1(\mathbf{x}) = x_0^2 + x_1^2 + x_2^2 + \ldots + x_7^2 = 1, \\
& H_2(\mathbf{x}) = 2(x_0 x_1 + x_1 x_2 + \ldots + x_{6} x_7) = a, \\
& G_4(\mathbf{x}) = 2(x_0 x_5 + x_1 x_6+x_2x_7 )=0.
\end{aligned}
\]

The third subcase is when $u_1=0, u_2\neq 0$, the minimization problem corresponding to this subcase is

\[
\begin{aligned}
\text{$S_{2,3}(a)$: Minimize } & F(\mathbf{x}) = (x_0 + x_1 + \ldots + x_7)^2 - 1 \\
\text{subject to } & H_1(\mathbf{x}) = x_0^2 + x_1^2 + x_2^2 + \ldots + x_7^2 = 1, \\
& H_2(\mathbf{x}) = 2(x_0 x_1 + x_1 x_2 + \ldots + x_{6} x_7) = a, \\
& G_5(\mathbf{x}) = 2(x_0 x_6 + x_1 x_7)=0.
\end{aligned}
\]

The fourth subcase is when $u_1\neq 0, u_2\neq 0$, the minimization problem corresponding to this subcase is

\[
\begin{aligned}
\text{$S_{2,4}(a)$: Minimize } & F(\mathbf{x}) = (x_0 + x_1 + \ldots + x_7)^2 - 1 \\
\text{subject to } & H_1(\mathbf{x}) = x_0^2 + x_1^2 + x_2^2 + \ldots + x_7^2 = 1, \\
& H_2(\mathbf{x}) = 2(x_0 x_1 + x_1 x_2 + \ldots + x_{6} x_7) = a, \\
& G_4(\mathbf{x}) = 2(x_0 x_5 + x_1 x_6+x_2x_7 )=0, \\
& G_5(\mathbf{x}) = 2(x_0 x_6 + x_1 x_7)=0.
\end{aligned}
\]

By penalty function method (Theorem \ref{th3.4}), for $i=1,2,3,4$, we have
\[
s_{2,i}(a) = \lim_{\mu \to \infty} \theta_i(\mu),
\]
where 
\begin{align*}
\theta_1(\mu) &= \inf \left\{ F(\mathbf{x}) + \mu \left((H_1(\mathbf{x})-1)^2+(H_2(\mathbf{x})-a)^2\right) \mid \mathbf{x} \in \mathbb{R}^8 \right\}, \\
\theta_2(\mu) &= \inf \left\{ F(\mathbf{x}) + \mu \left((H_1(\mathbf{x})-1)^2+(H_2(\mathbf{x})-a)^2+G_4(\mathbf{x})^2\right) \mid \mathbf{x} \in \mathbb{R}^8 \right\},\\
\theta_3(\mu) &= \inf \left\{ F(\mathbf{x}) + \mu \left((H_1(\mathbf{x})-1)^2+(H_2(\mathbf{x})-a)^2+G_5(\mathbf{x})^2\right) \mid \mathbf{x} \in \mathbb{R}^8 \right\},\\
\theta_4(\mu) &= \inf \left\{ F(\mathbf{x}) + \mu \left((H_1(\mathbf{x})-1)^2+(H_2(\mathbf{x})-a)^2+\sum_{i=1}^2G_{3+i}(\mathbf{x})^2\right) \mid \mathbf{x} \in \mathbb{R}^8 \right\}.
\end{align*}
A Matlab-based algorithm is subsequently used to calculate $\di \frac{s_{2,i}(a)}{(\sqrt{a}-1)^2}$ for $i=1,2,3,4$, by adopting the Newton-Raphson method in conjunction with a penalty function technique to iteratively find the minimizer of the penalized objective function (see \ref{Appendix B} for the detailed Matlab implementation).

For each subproblem \( S_{2,i}(a) \), a graph (Figure\ref{fig14}, Figure\ref{fig15}, Figure\ref{fig16} and Figure\ref{fig17}) is generated plotting the ratio \(\di \frac{s_{2,i}(a)}{(\sqrt{a} - 1)^2} \) against \( a \). In these graphs, different colors are used to show feasible and infeasible solutions to the original problem $S_2(a)$. Red points indicate infeasible solutions that solve the subproblems but do not meet all constraints of $S_2(a)$. On the other hand, green points represent feasible solutions that satisfy all the constraints.
\begin{figure}[H]
  \centering
  \includegraphics[width=9cm]{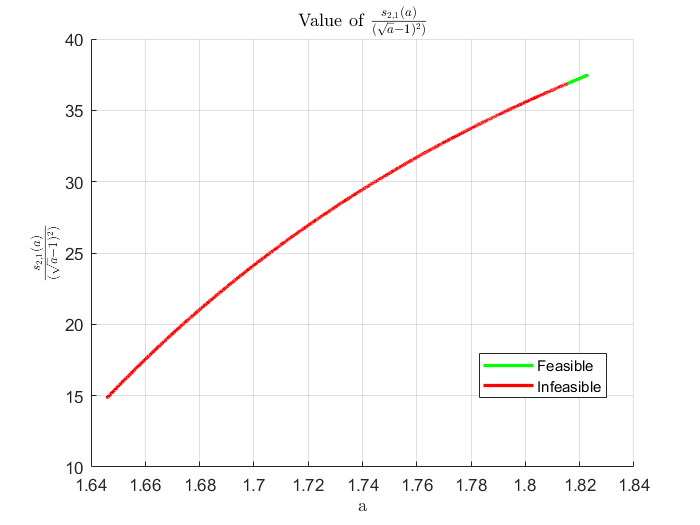}
  \caption{The ratio $\di \frac{s_{2,1}(a)}{(\sqrt{a}-1)^2}$.}
  \label{fig14}
\end{figure}

  \begin{figure}[H]
    \centering
    \includegraphics[width=9cm]{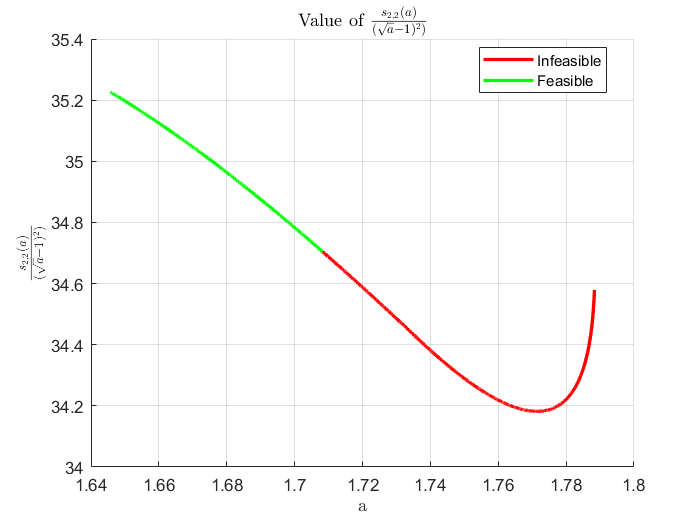}
    \caption{The ratio $\di \frac{s_{2,2}(a)}{(\sqrt{a}-1)^2}$.}
    \label{fig15}
    \end{figure}

    \begin{figure}[H]
      \centering
      \includegraphics[width=9cm]{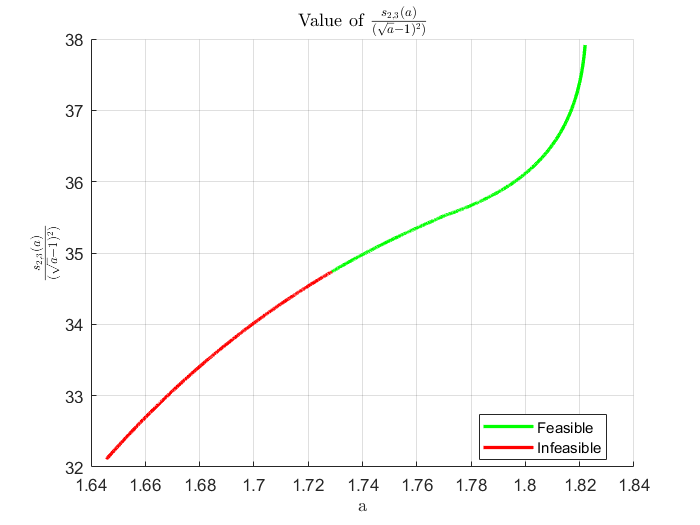}
      \caption{The ratio $\di \frac{s_{2,3}(a)}{(\sqrt{a}-1)^2}$.}
      \label{fig16}
      \end{figure}

\begin{figure}[H]
      \centering
        \includegraphics[width=9cm]{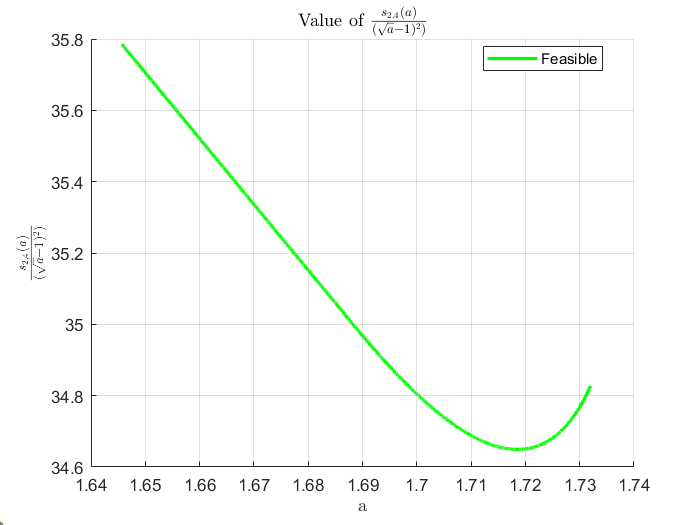}
        \caption{The ratio $\di \frac{s_{2,4}(a)}{(\sqrt{a}-1)^2}$.}
        \label{fig17}
        \end{figure}
      
Among all feasible solutions for each \( a \), \( s_2(a) \) is identified as the minimum value of all $s_{2,i}(a)$ that are feasible to the problem $S_2(a)$. Figure \ref{fig18} shows the graph plotting the ratio of the function $\di\frac{s_{2}(a)}{(\sqrt{a}-1)^2}$.
\begin{figure}[H]
  \centering
  \includegraphics[width=9cm]{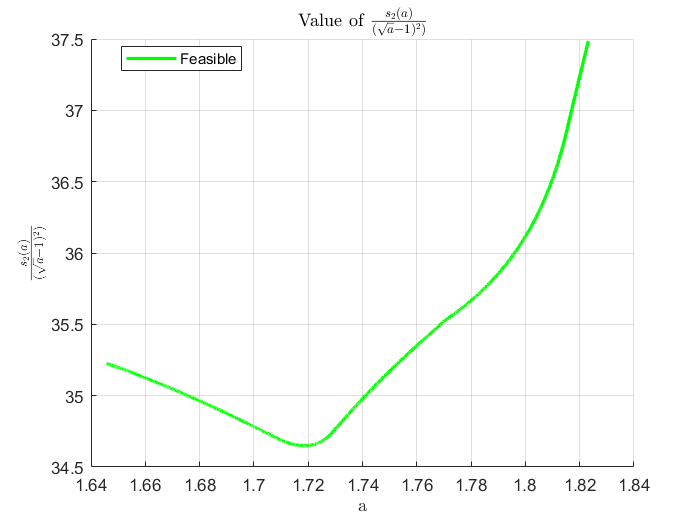}
  \caption{The ratio $\di \frac{s_{2}(a)}{(\sqrt{a}-1)^2}$.}
  \label{fig18}
  \end{figure}

Then, the infimum of the ratio is determined to be \( 34.64948641 \), achieved at a specific point in \( \mathbb{R}^8 \). This point is given by
\begin{align*}
  s&=[0.1685903,0.4506317,0.6267577,0.5454191,\\
  &0.2756348,0.0361770,-0.0282171,0.1055656].
  \end{align*}

 The polynomial corresponding to this solution is,
$$f(\phi)=a_0+a_1\cos\phi+a_2\cos 2\phi+a_3\cos 3\phi+a_4\cos 4\phi+a_5 \cos 5\phi+a_6\cos 6\phi,$$
where
\begin{align*}
a_0&=1,\\
a_1&=1.7185098,\\ 
a_2&=1.0731034,\\
a_3&=0.4527292,\\
a_4&=0.101695001,\\
a_5&=0,\\
a_6&=0,\\
a_7&=0.003559468.
\end{align*}

One could see that this point $s$ also lies in the feasible region of $S_1\left(a_1\right)$. This implies that the optimal solution of problem $S_2\left(a_1\right)$ is also the optimal solution of problem $S_1\left(a_1\right)$ and thus agrees with the value of $\chi_7\left(a_1\right)$. By the relationship
$$
\frac{\chi_7(a)}{(\sqrt{a}-1)^2} \geq \frac{s_2(a)}{(\sqrt{a}-1)^2},
$$
one can conclude
\begin{align*}
V_7&=\inf_{a \in \di [1.6456659,1.8231801]} \frac{\chi_7(a)}{(\sqrt{a} - 1)^2}\\
&= \inf_{a \in \di [1.6456659,1.8231801]} \frac{s_2(a)}{(\sqrt{a} - 1)^2}\\
&= 34.6494874.
\end{align*}
This completes the proof.
\end{proof}

\vfil\pagebreak

\section{Exact value of $V_8$.}
In this section, we aim to determine the exact value of $V_8$, following the same approach as in the previous sections. From Equation (\ref{eq1.2}), the value of $V_8$ is given by  
$$
V_8=\inf _{a \in\left(1, 2\di\cos\frac{\pi}{10}\right]} \frac{\chi_8(a)}{(\sqrt{a}-1)^2},
$$
where
$$\chi_8(a)=\inf \left\{f(0)-1 \mid f(\phi)\in C_8, a_0\left(f\right)=1 \text { and } a_1\left(f\right)=a\right\} \text {. }$$

\subsection{Exact value of $V_8$.}

In order to find the exact value of $V_8$, we will use the Karush-Kuhn-Tucker conditions and the penalty function method to solve the optimization problem $\chi_8(a)$. To begin, we will narrow the range of a to a smaller interval.

By a similar proof in Theorem \ref{th6.1}, we have the following theorem.
\begin{theorem} The exact value of $V_8$ could be obtained by
$$V_8=\inf _{a \in[1.6566924,1.8191095]} \frac{\chi_8(a)}{(\sqrt{a}-1)^2}.$$
\end{theorem}

\begin{proof} Since the set of cosine polynomials $C_7 \subset C_8$, one have $V_8\leq V_7=34.6494874$. However, by Theorem $\ref{th4.1}$ and Theorem $\ref{th4.2}$, we have
\begin{align}
V_8(a) &\geq F_1(a)=\frac{2a-1}{(\sqrt{a}-1)^2},\\
V_8(a) &\geq F_2(a)=\frac{5.87267810 a-6.87267810}{(\sqrt{a}-1)^2},\\
V_8(a) &\geq F_3(a)=\frac{ 16.5 a-25.8011608}{(\sqrt{a}-1)^2}.
\end{align}
for $a \in \left(1,2\di\cos\frac{\pi}{10}\right].$ Given function $F:(1,\infty)\to \mathbb{R}$ with the form
$$F(a)=\frac{Aa-B}{(\sqrt{a}-1)^2},$$
where $A,B$ are constant. Taking the derivative, one has
$$F'(a)=\frac{B-A \sqrt{x}}{(\sqrt{x}-1)^3 \sqrt{x}}.$$
This implies when $A>B$ the function $F(a)$ is decreasing on the interval $(1,\infty)$. On the other hand, when $A<B$, the function increasing on the interval $\di \left(1,(B/A)^2\right]$ and decreasing on the interval $\di[(B/A)^2,\infty)$. Apply this on the function $F_1(a),F_2(a)$ and $F_3(a)$, we found out $F_1(a)$ is a decreasing function on the interval $\di\left(1,2\cos\frac{\pi}{10}\right]$. However, the function $F_2(a)$ is increasing on $(1,1.3695543]$ and is decreasing on the interval $\di\left[1.36955543,2\cos\frac{\pi}{10}\right]$. On the same time the function $F_3(a)$ is increasing on the interval $\di\left(1,2\cos\frac{\pi}{10}\right]$.

Now, given any $f(\phi) \in C_8$ with $a_0(f)=1,a_1(f)=a\in (1,1.5542038]$, one has
$$v(f)\geq \frac{\chi(a)}{(\sqrt{a}-1)^2} \geq F_1(a) \geq F_1(1.5542038)=34.6494937>V_8. $$
Similarly, for any $f(\phi) \in C_8$ with $a_0(f)=1,a_1(f)=a\in [1.5542038,1.6566924]$, one could obtain
$$v(f)\geq \frac{\chi(a)}{(\sqrt{a}-1)^2} \geq F_2(a) \geq  F_2(1.6566924)=34.6494895>V_8. $$
Finally, for any $f(\phi) \in C_8$ with $a_0(f)=1,a_1(f)=a\in \left[1.8191095,2\cos\frac{\pi}{10}\right]$, one could obtain
$$v(f)\geq \frac{\chi(a)}{(\sqrt{a}-1)^2} \geq F_3(a) \geq  F_3(1.8191095)=34.6494933>V_8. $$
This completes the proof.
\end{proof}

By Fejer-Reisz theorem, the problem of finding the $\chi_8(a)$ for $\di a\in\left(1, 2\di\cos\frac{\pi}{10}\right]$ can be formalized as the minimization problem in $\mathbb{R}^9$ as:

\begin{problem}

 Consider the optimization problem \( T_1(a) \) defined as follows:

\[
\begin{aligned}
\text{$T_1(a)$: Minimize } & F(\mathbf{x}) = (x_0 + x_1 + \ldots + x_8)^2 - 1 \\
\text{subject to } & H_1(\mathbf{x}) = x_0^2 + x_1^2 + x_2^2 + \ldots + x_8^2 = 1, \\
& H_2(\mathbf{x}) = 2(x_0 x_1 + x_1 x_2 + \ldots + x_{7} x_8) = a, \\
& G_1(\mathbf{x}) = 2(x_0 x_2 + x_1 x_3 + \ldots+ x_{6} x_8) \geq 0, \\
& G_2(\mathbf{x}) = 2(x_0 x_3 + x_1 x_4 + \ldots+ x_{5} x_8) \geq 0, \\
& \quad\quad\quad{\vdots} \\
& G_{7}(\mathbf{x}) = 2x_0 x_8 \geq 0.
\end{aligned}
\]

Define \( t_1(a) \) as the optimal value of the objective function \( F(\mathbf{x}) \) for a given \( a \):

\[
t_1(a) = \min_{\mathbf{x}} F(\mathbf{x}) \quad \text{subject to the constraints of } T_1(a).
\]
\end{problem}

However, instead of considering Problem $T_1(a)$ with nine constraints, we will consider the following problem with four constraints.

\begin{problem}
Consider the optimization problem \( T_2(a) \) defined as follows:
\[
\begin{aligned}
\text{$T_2(a)$: Minimize } & F(\mathbf{x}) = (x_0 + x_1 + \ldots + x_8)^2 - 1 \\
\text{subject to } & H_1(\mathbf{x}) = x_0^2 + x_1^2 + x_2^2 + \ldots + x_8^2 = 1, \\
& H_2(\mathbf{x}) = 2(x_0 x_1 + x_1 x_2 + \ldots + x_{7} x_8) = a, \\
& G_4(\mathbf{x}) = 2(x_0 x_5 + x_1 x_6+x_2x_7+x_3x_8 ) \geq 0, \\
& G_5(\mathbf{x}) = 2(x_0 x_6 + x_1 x_7+x_2x_8) \geq 0.
\end{aligned}
\]

Define \( t_2(a) \) as the optimal value of the objective function \( F(\mathbf{x}) \) for a given \( a \):

\[
t_2(a) = \min_{\mathbf{x}} F(\mathbf{x}) \quad \text{subject to the constraints of } T_2(a).
\]
\end{problem}

Since the feasible region of problem $T_1(a)$ is the subset of the feasible region of problem $T_2(a)$, this implies
$$\chi_8(a)=t_1(a)\geq t_2(a).$$
If the optimal point of problem $T_2(a)$ also satisfies the condition $G_1(\mathbf{x}), G_2(\mathbf{x}),$ $ G_3(\mathbf{x}), G_6(\mathbf{x}),$ and $G_7(\mathbf{x})$, then the optimal point will lie inside the feasible region of problem $T_1(a)$. This implies it will also be the solution of problem $T_1(a)$. In this case,
$$\chi_8(a)=t_1(a)=t_2(a).$$
\begin{theorem}
The exact value of \( V_8 \) can be expressed as:

\[
V_8=\inf_{a \in[1.6566924,1.8191095]} \frac{t_2(a)}{(\sqrt{a} - 1)^2} = 34.5399155.
\]

The exact solution corresponding to this infimum could be obtained at the point
\begin{align*}
t=&[0.1246536,0.3805581,0.5968565,0.5888198,0.3562317,\\
&0.0912429,-0.0315349,-0.0146819,0.0159473]
\end{align*}
The polynomial corresponding to this solution is,
\begin{align*}
f(\phi)&=a_0+a_1\cos\phi+a_2\cos 2\phi+a_3\cos 3\phi+a_4\cos 4\phi\\ &+a_5 \cos 5\phi+a_6\cos 6\phi+a_7\cos 7\phi+a_8\cos 8\phi,
\end{align*}
where
\begin{align*}
a_0&=1,\\
a_1&=1.7312576,\\ 
a_2&=1.1034980,\\
a_3&=0.4821616,\\
a_4&=0.1146858,\\
a_5&=0,\\
a_6&=0,\\
a_7&=0.0084774,\\
a_8&=0.0039758.
\end{align*}

\end{theorem}

\begin{proof}
Fix an $a\in [1.6566924,1.8191095]$, we would first like to establish the numerical solution of the optimization problem \( T_2(a) \). By applying  the Karush-Kuhn-Tucker conditions (Theorem \ref{th3.3}), if  a point $\overline{\mathbf{x}}\in \mathbb{R}^9$ solves Problem $T_2(a)$ globally, there exist scalars $\lambda_1,\lambda_2,u_1,u_2$  such that

\begin{align}
\nabla F(\overline{\mathbf{x}})+\sum_{i=1}^2 \lambda_i \nabla H_i(\overline{\mathbf{x}})+\sum_{i=1}^2 u_i \nabla G_{3+i}(\overline{\mathbf{x}}) & =\mathbf{0} & \\
u_i G_{3+i}(\overline{\mathbf{x}}) & =0 & \text { for } i=1, 2\label{eq8} \\
u_i & \geq 0 & \text { for } i=1, 2
\end{align}

By the Equation (\ref{eq8}), the problem $T_2(a)$ could be divided into the following four subproblems, we call them problem $T_{2,1}(a), T_{2,2}(a), T_{2,3}(a)$ and $T_{2,4}(a)$ respectively. We will use $t_{2,1}(a),t_{2,2}(a),t_{2,3}(a)$ and $t_{2,4}(a)$ to denote the solutions of these subproblems respectively. The first subcase is when $u_1=0, u_2=0$, the minimization problem corresponding to this subcase is

\[
\begin{aligned}
\text{$T_{2,1}(a)$: Minimize } & F(\mathbf{x}) = (x_0 + x_1 + \ldots + x_8)^2 - 1 \\
\text{subject to } & H_1(\mathbf{x}) = x_0^2 + x_1^2 + x_2^2 + \ldots + x_8^2 = 1, \\
& H_2(\mathbf{x}) = 2(x_0 x_1 + x_1 x_2 + \ldots + x_{7} x_8) = a. 
\end{aligned}
\]

The second subcase is when $u_1\neq 0, u_2=0$, the minimization problem corresponding to this subcase is    

\[ 
\begin{aligned} 
\text{$T_{2,2}(a)$: Minimize } & F(\mathbf{x}) = (x_0 + x_1 + \ldots + x_8)^2 - 1 \\
\text{subject to } & H_1(\mathbf{x}) = x_0^2 + x_1^2 + x_2^2 + \ldots + x_8^2 = 1, \\
& H_2(\mathbf{x}) = 2(x_0 x_1 + x_1 x_2 + \ldots + x_{7} x_8) = a, \\
& G_4(\mathbf{x}) = 2(x_0 x_5 + x_1 x_6+x_2x_7+x_3x_8 )=0.
\end{aligned}
\]

The third subcase is when $u_1=0, u_2\neq 0$, the minimization problem corresponding to this subcase is

\[
\begin{aligned}
\text{$T_{2,3}(a)$: Minimize } & F(\mathbf{x}) = (x_0 + x_1 + \ldots + x_8)^2 - 1 \\
\text{subject to } & H_1(\mathbf{x}) = x_0^2 + x_1^2 + x_2^2 + \ldots + x_8^2 = 1, \\
& H_2(\mathbf{x}) = 2(x_0 x_1 + x_1 x_2 + \ldots + x_{7} x_8) = a, \\
& G_5(\mathbf{x}) = 2(x_0 x_6 + x_1 x_7+x_2x_8)=0.
\end{aligned}
\]

The fourth subcase is when $u_1\neq 0, u_2\neq 0$, the minimization problem corresponding to this subcase is

\[
\begin{aligned}
\text{$T_{2,4}(a)$: Minimize } & F(\mathbf{x}) = (x_0 + x_1 + \ldots + x_8)^2 - 1 \\
\text{subject to } & H_1(\mathbf{x}) = x_0^2 + x_1^2 + x_2^2 + \ldots + x_8^2 = 1, \\
& H_2(\mathbf{x}) = 2(x_0 x_1 + x_1 x_2 + \ldots + x_{7} x_8) = a, \\
& G_4(\mathbf{x}) = 2(x_0 x_5 + x_1 x_6+x_2x_7+x_3x_8 )=0, \\
& G_5(\mathbf{x}) = 2(x_0 x_6 + x_1 x_7+x_2x_8)=0.
\end{aligned}
\]

By penalty function method (Theorem \ref{th3.4}), for $i=1,2,3,4$, we have
\[
t_{2,i}(a) = \lim_{\mu \to \infty} \theta_i(\mu),
\]
where 
\begin{align*}
\theta_1(\mu) &= \inf \left\{ F(\mathbf{x}) + \mu \left((H_1(\mathbf{x})-1)^2+(H_2(\mathbf{x})-a)^2\right) \mid \mathbf{x} \in \mathbb{R}^9 \right\}, \\
\theta_2(\mu) &= \inf \left\{ F(\mathbf{x}) + \mu \left((H_1(\mathbf{x})-1)^2+(H_2(\mathbf{x})-a)^2+G_4(\mathbf{x})^2\right) \mid \mathbf{x} \in \mathbb{R}^9 \right\},\\
\theta_3(\mu) &= \inf \left\{ F(\mathbf{x}) + \mu \left((H_1(\mathbf{x})-1)^2+(H_2(\mathbf{x})-a)^2+G_5(\mathbf{x})^2\right) \mid \mathbf{x} \in \mathbb{R}^9 \right\},\\
\theta_4(\mu) &= \inf \left\{ F(\mathbf{x}) + \mu \left((H_1(\mathbf{x})-1)^2+(H_2(\mathbf{x})-a)^2+\sum_{i=1}^2G_{3+i}(\mathbf{x})^2\right) \mid \mathbf{x} \in \mathbb{R}^9 \right\}.
\end{align*}

Then, a Matlab-based algorithm is employed to compute the $\di \frac{t_{2,i}(a)}{(\sqrt{a}-1)^2}$ for $i=1,2,3,4$, through Newton-Raphson method and penalty function approach.

For each subproblem \( t_{2,i}(a) \), a graph (Figure \ref{fig19}, Figure \ref{fig20}, Figure \ref{fig21} and Figure \ref{fig22}) is generated plotting the ratio \(\di \frac{t_{2,i}(a)}{(\sqrt{a} - 1)^2} \) against \( a \). Colors are utilized in the graphs to differentiate between feasible and infeasible solutions for the original problem $T_2(a)$. Red points signify infeasible solutions that resolve subproblems but violate certain constraints, while green points represent feasible solutions that adhere to all constraints.
\begin{figure}[H]
  \centering
  \includegraphics[width=9cm]{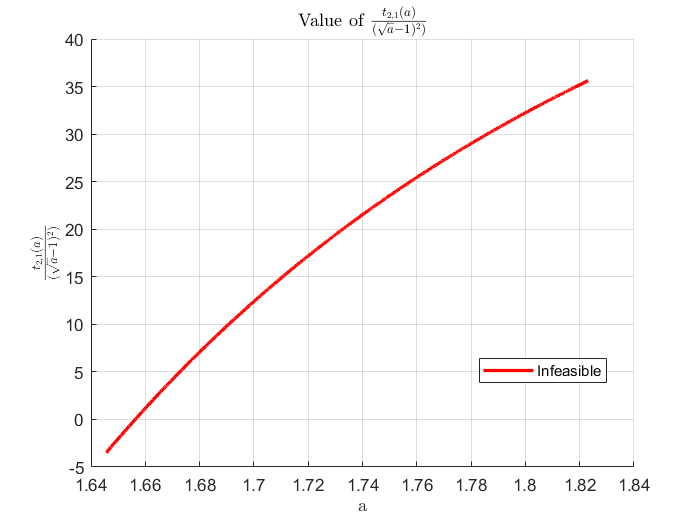}
  \caption{The ratio $\di \frac{t_{2,1}(a)}{(\sqrt{a}-1)^2}$.}
  \label{fig19}
  \end{figure}

  \begin{figure}[H]
    \centering
    \includegraphics[width=9cm]{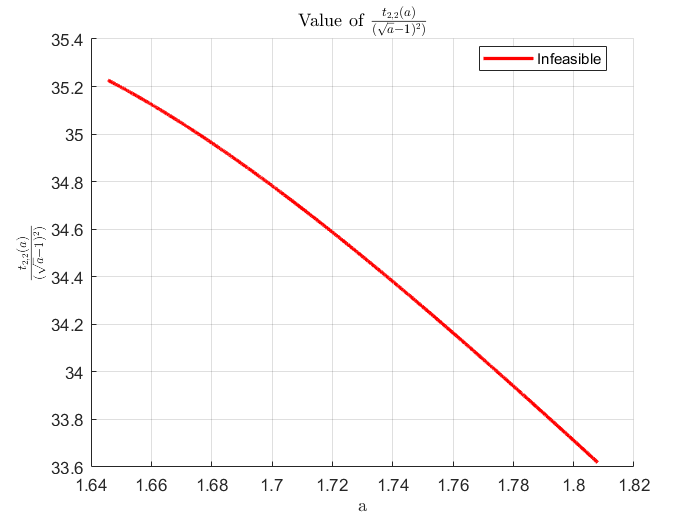}
    \caption{The ratio $\di \frac{t_{2,2}(a)}{(\sqrt{a}-1)^2}$.}
    \label{fig20}
    \end{figure}

    \begin{figure}[H]
      \centering
      \includegraphics[width=9cm]{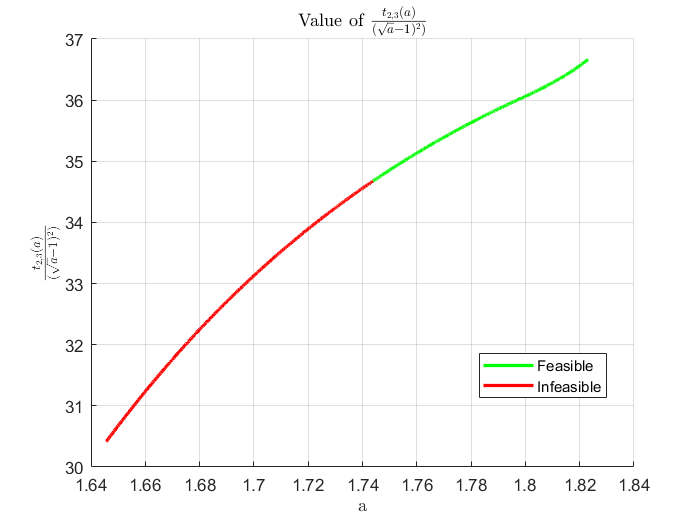}
      \caption{The ratio $\di \frac{t_{2,3}(a)}{(\sqrt{a}-1)^2}$.}
      \label{fig21}
      \end{figure}

\begin{figure}[H]
        \centering
        \includegraphics[width=9cm]{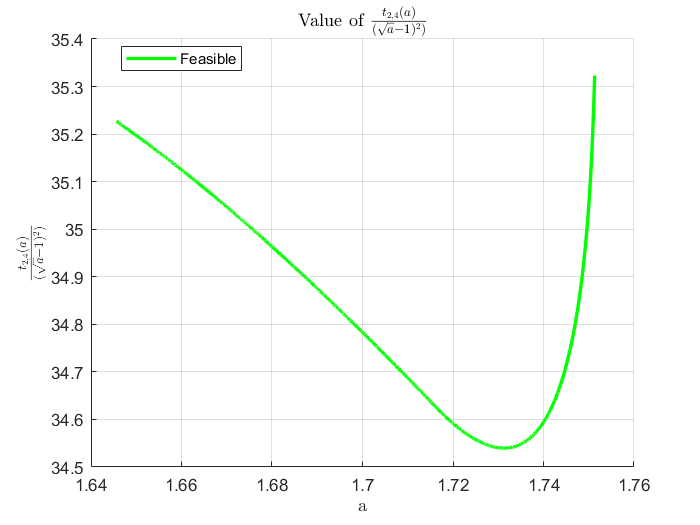}
        \caption{The ratio $\di \frac{t_{2,4}(a)}{(\sqrt{a}-1)^2}$.}
        \label{fig22}
        \end{figure}
      
Among all feasible solutions for each \( a \), \( t_2(a) \) is identified as the minimum value of all $t_{2,i}(a)$ that are feasible to the problem $T_2(a)$. The following graph (Figure \ref{fig23}) is generated plotting the ratio of the function $\di\frac{t_{2}(a)}{(\sqrt{a}-1)^2}$.
\begin{figure}[H]
  \centering
  \includegraphics[width=9cm]{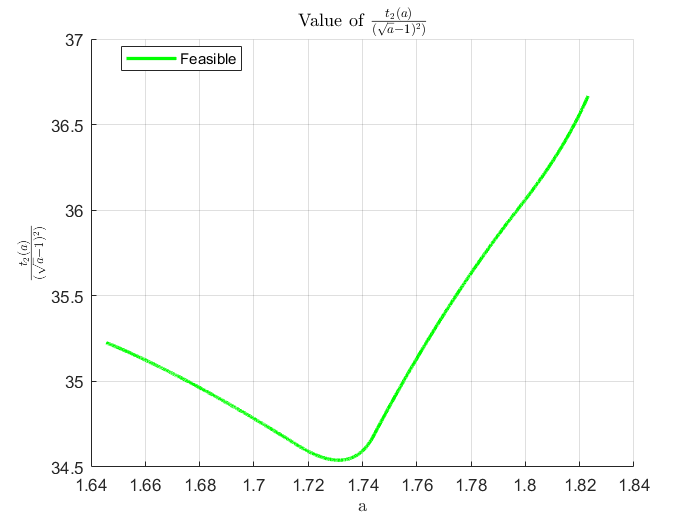}
  \caption{The ratio $\di \frac{t_{2}(a)}{(\sqrt{a}-1)^2}$.}
  \label{fig23}
  \end{figure}

Then, the infimum of the ratio is determined to be \( 34.5399155 \), achieved at  

\begin{align*}
  t=&[0.1246536,0.3805581,0.5968565,0.5888198,0.3562317,\\
  &0.0912429,-0.0315349,-0.0146819,0.0159473].
  \end{align*}
  The polynomial corresponding to this solution is,
  \begin{align*}
  f(\phi)&=a_0+a_1\cos\phi+a_2\cos 2\phi+a_3\cos 3\phi+a_4\cos 4\phi\\ &+a_5 \cos 5\phi+a_6\cos 6\phi+a_7\cos 7\phi+a_8\cos 8\phi,
  \end{align*}
  where
  \begin{align*}
  a_0&=1,\\
  a_1&=1.7312576,\\ 
  a_2&=1.1034980,\\
  a_3&=0.4821616,\\
  a_4&=0.1146858,\\
  a_5&=0,\\
  a_6&=0,\\
  a_7&=0.0084774,\\
  a_8&=0.0039758.
  \end{align*}

One could see that this point $t$ also lies in the feasible region of $T_1\left(a_1\right)$. This implies that the optimal solution of problem $T_2\left(a_1\right)$ is also the optimal solution of problem $T_1\left(a_1\right)$ and thus agrees with the value of $\chi_8\left(a_1\right)$. By the relationship
$$
\frac{\chi_8(a)}{(\sqrt{a}-1)^2} \geq \frac{t_2(a)}{(\sqrt{a}-1)^2},
$$
one can conclude
\begin{align*}
V_8&=\inf_{a \in \di [1.6566924,1.8191095]} \frac{\chi_8(a)}{(\sqrt{a} - 1)^2}\\
&= \inf_{a \in \di [1.6566924,1.8191095]} \frac{t_2(a)}{(\sqrt{a} - 1)^2}\\
&= 34.5399155.
\end{align*}

This completes the proof.

\end{proof}

\vfil\pagebreak
\section*{Acknowledgements}
The work is supported by the Xiamen University Malaysia Research Fund (XMUMRF/2021-C8/IMAT/0017).

This work continues the author's undergraduate thesis, which was supervised by Prof. Dr. Teo Lee Peng at Xiamen University Malaysia. The author expresses sincere gratitude for her guidance, which lasted approximately 18 months. He appreciates all the knowledge and qualities he learned from her. Additionally, the author would like to thank the Department of Mathematics and Applied Mathematics at Xiamen University Malaysia for providing a quality education.

Last but not least, the support from family members, his partner, and every teacher was the only reason the author was able to overcome all hardships and turn professional.

\vfil\pagebreak
\appendix

\section{Proof of the lemma in Section 2.3.}\label{Appendix A}
\begin{lemma}[Lemma 2.9]\label{A.1}
Given four positive real numbers $a, b, \Delta a, \Delta b$, if
\[
\frac{\Delta a}{\Delta b} > \frac{a}{b},
\]
then
\[
\frac{a + \Delta a}{b + \Delta b} > \frac{a}{b}.
\]
\end{lemma}

\begin{proof}
Let $\di\frac{a}{b} = k$, then $\di\Delta a > \frac{a}{b} \Delta b = k \Delta b$, which implies
\[
\frac{a + \Delta a}{b + \Delta b} > \frac{kb + k\Delta b}{b + \Delta b} = k.
\]
This completes the proof.
\end{proof}

\begin{lemma}[Lemma 2.10] \label{A.2}
Given five positive real numbers $a_1, a_2, b_1, b_2, k$, if
\[
\frac{a_1}{b_1} > k \quad \text{and} \quad \frac{a_2}{b_2} > k,
\]
then
\[
\frac{a_1 + a_2}{b_1 + b_2} > k.
\]
\end{lemma}

\begin{proof}
This is just an implication of Lemma \ref{A.1}. Without loss of generality, assume that $\di \frac{a_1}{b_1} \geq \frac{a_2}{b_2}$, then
\[
\frac{a_1 + a_2}{b_1 + b_2} \geq \frac{a_2}{b_2} > k.
\]
This completes the proof.
\end{proof}

\begin{lemma}[Lemma 2.11]
Given four positive real numbers $a, b, \Delta a, \Delta b$, then
\[
\sqrt{(a + \Delta a)(b + \Delta b)} \geq \sqrt{ab} + \sqrt{\Delta a \Delta b},
\]
equality holds when $a \Delta b = b \Delta a.$
\end{lemma}

\begin{proof}
Since both sides of the inequality are positive, it is enough to prove
\[
(a + \Delta a)(b + \Delta b) \geq ab + \Delta a \Delta b + 2\sqrt{ab \Delta a \Delta b}.
\]
Simplifying, we get
\[
a \Delta b + b \Delta a \geq 2\sqrt{ab \Delta a \Delta b}.
\]
The arithmetic-geometric mean inequality implies the inequality above is true, and equality holds when $a \Delta b = b \Delta a$.
\end{proof}

\vfil\pagebreak

\section{Matlab Implementation for Computing \( Q_{2,1}(a) \)}
\label{Appendix B}

This appendix provides the Matlab code utilized to solve the optimization problem \( Q_{2,1}(a) \) for determining the exact value of \( V_5 \). The code employs the Newton-Raphson method in conjunction with a penalty function approach to iteratively find the minimizer of the penalized objective function. The code could be modified to solve the problem $P_2(a), Q_{2,2}(a), R_{2,i}(a), S_{2,i}(a)$, and $T_{2,i}(a)$ for $i=1,2,3,4$.

\begin{lstlisting}[style=MATLABStyle, caption={MATLAB Code for computing \( q_{2,1}(a) \)}, label={lst:V4_Code}]
  
  function RESULT=V5()
  VValue_ANS=[0]
  AValue=[0]
  Feasible=[0]
  A=1.6456659
  B=2*cos(pi/7)
  P=10^7
  
  for n=0:P
  
  a=A+(B-A)*n/P

  % For different values of a, the initial guess is different to address the issue of local solution and convergence.

  if a>=1.6456 && a< 1.68
      V5=Algorithm(a,[0.2813599288; 0.5616322755; 0.6297662332; 0.4344555262; 0.1227731004; -0.0705778044],n)
  end
  
  if a>=1.68 && a< 1.72
      V5=Algorithm(a,[-0.0201966983; 0.1955685973; 0.4848202185; 0.6297457947; 0.5212134113; 0.2409501665],n)
  end
  
  if a>=1.72 && a< 1.76
      V5=Algorithm(a,[-0.1711276227; -0.4436859255; -0.6124330758; -0.5502991243; -0.3040722037; -0.0591661108 ],n)
  end
  
  if a>=1.76 && a< 1.81
      V5=Algorithm(a,[-0.1579917997; -0.3946171069; -0.5650843883; -0.5650843821; -0.3946170930; -0.1579917886],n)
  end
  
  % Use two while loops to ensure the algorithm converges to the global minimum.

  while abs(V5(1)-V5(2))>=10^(-7) || abs(V5(3)-1)>=10^(-7)  || abs(V5(4)-a)>=10^(-7) 
    
          V5=Algorithm(a,[2*rand(1)-1 ; 2*rand(1)-1 ; 2*rand(1)-1 ; 2*rand(1)-1 ; 2*rand(1)-1 ;2*rand(1)-1 ],n)
              
  end
  
  if n>=1
     while abs(V5(1)-V5(2))>=10^(-7) || abs(V5(3)-1)>=10^(-7) || abs(V5(1)-VValue_ANS(n))>10^{-7}
          V5=Algorithm(a,[ 2*rand(1)-1 ; 2*rand(1)-1 ; 2*rand(1)-1 ; 2*rand(1)-1 ; 2*rand(1)-1 ;2*rand(1)-1 ],n)
   
     end
  end
  
  % Check if the solution is feasible to the problem $Q_2(a)$.
  if  V5(8)>=-0.001
      feasible=1
  else
      feasible=0
  end
  
  Feasible(n+1)=feasible
  AValue(n+1)=a
  VValue_ANS(n+1)=V5(1,1)
  RESULT=[AValue;VValue_ANS;Feasible]
  end
  
  X=RESULT(1,:)
  Y=RESULT(2,:)
  Feasible=RESULT(3,:)
  
  % Create a new figure
  figure;
  
  % Define marker size
  markerSize = 5; % Adjust the size as needed to make the curve appear smooth
  
  % Plot points where Feasible == 0 with smaller filled red circles
  
  scatter(X(Feasible== 1), Y(Feasible== 1), markerSize, 'filled', 'MarkerFaceColor', 'g');
  
  hold on;  % Hold on to the current plot
  scatter(X(Feasible== 0), Y(Feasible== 0), markerSize, 'filled', 'MarkerFaceColor', 'r');
  
  
  % Adding labels and title for better information
  xlabel('a', 'Interpreter', 'latex');
  ylabel('$\frac{q_{2,1}(a)}{(\sqrt{a}-1)^2)}$', 'Interpreter', 'latex');
  title('Value of $\frac{q_{2,1}(a)}{(\sqrt{a}-1)^2)}$', 'Interpreter', 'latex');
  
  % Adding grid for better visualization
  grid on;
  
  % Hold off to stop adding to the current plot
  hold off;
  
  
  end
  
  
  function x_opt= Algorithm(a,xstart,n)
  %Test of V5
  % increases the r iteratively to ensure the convergence to the global minimum.
  r=100000000
  % Define the function and its gradient
  Vvalue=1000
  
  syms x0 x1 x2 x3 x4 x5 ;

  % Define the constant and whether we need to consider it in the optimization problem.
  h5=2*x0*x5;
  I1=0
  
  % The main penalty function.
  f=(x0+x1+x2+x3+x4+x5)^2+r*(x0^2+x1^2+x2^2+x3^2+x4^2+x5^2-1)^2+r*(2*(x0*x1+x1*x2+x2*x3+x3*x4+x4*x5)- a)^2+r*(I1*h5)^2; 

  % Define the gradient of the function
  grad_f = gradient(f, [x0, x1, x2, x3, x4,x5]);
  
  % Define the Hessian matrix
  hessian_f = hessian(f, [x0, x1, x2, x3, x4,x5]);
  
  % Initial guess for the minimum
  % Newton's method iterations
  max_iterations = 10000000;
  tolerance = 10^(-4);
  for iter = 1:max_iterations
      % Evaluate the function and its gradient at the current point
      f_val = double(subs(f, [x0, x1, x2, x3, x4,x5], xstart'));
      grad_val = double(subs(grad_f, [x0, x1, x2, x3, x4,x5], xstart'));
  
      % Check convergence
      if norm(grad_val) < tolerance
          
  fprintf('Converged to a minimum: [%.10f; %.10f; %.10f; %.10f; %.10f; %.10f]\n', xstart(1), xstart(2), xstart(3),xstart(4),xstart(5),xstart(6));
          
          break;
      end
      
      % Calculate the Newton update direction
      hessian_val = double(subs(hessian_f, [x0, x1, x2, x3, x4,x5], xstart'));
      delta_x = -inv(hessian_val) * grad_val;
      
      % Update the current point
      xstart = xstart + delta_x;
      
      % Display iteration information
  fprintf('Iteration %d: [%.8f, %.8f, %.8f, %.8f, %.8f, %.8f], Gradient norm: %f\n', iter, xstart(1), xstart(2), xstart(3),xstart(4), xstart(5), xstart(6), norm(grad_val));
  
  % Calculate the corresponding cosine polynomial.
  a0=xstart(1)^2+xstart(2)^2+xstart(3)^2+xstart(4)^2+xstart(5)^2+xstart(6)^2;
  a1=2*(xstart(1)*xstart(2)+xstart(2)*xstart(3)+xstart(3)*xstart(4)+xstart(4)*xstart(5)+xstart(5)*xstart(6));
  a2=2*(xstart(1)*xstart(3)+xstart(2)*xstart(4)+xstart(3)*xstart(5)+xstart(4)*xstart(6));
  a3=2*(xstart(1)*xstart(4)+xstart(2)*xstart(5)+xstart(3)*xstart(6));
  a4=2*(xstart(1)*xstart(5)+xstart(2)*xstart(6));
  a5=2*( xstart(1)*xstart(6));
  
  % Calculate the value of the value of $v(f)$ in two different methods to ensure the convergence to the global minimum.
  Vvalue=(a1+a2+a3+a4+a5)/(sqrt(a1)-sqrt(a0))^2
  Vvalue2=((xstart(1)+xstart(2)+xstart(3)+xstart(4)+xstart(5)+xstart(6))^2-a0)/(sqrt(a1)-sqrt(a0))^2
  x_opt=[Vvalue,Vvalue2,a0,a1,a2,a3,a4,a5,n]
  
  end
  
  
  if iter == max_iterations
      fprintf('Maximum iterations reached without convergence\n');
      V5=[0,1,0,0,0,0,0,0,0]
  end
  
  end
\end{lstlisting}
\vfil\pagebreak

\bibliographystyle{plain} 
\bibliography{Optimal_Cosine_Polynomials_for_Riemann_Zeta_Zero-Free_Region.bib} 

\begin{thebibliography}{10}

\bibitem{ap1976}
T.~M. Apostol.
\newblock {\em Introduction to Analytic Number Theory}.
\newblock Springer New York, 1976.

\bibitem{ar92}
V.~V. Arestov.
\newblock On an extremal problem for nonnegative trigonometric polynomials.
\newblock {\em Trudy Instituta Matematiki i Mekhaniki}, 1:50--70, 1992.
\newblock [in Russian].

\bibitem{ar90}
V.~V. Arestov and V.~P. Kondrat’ev.
\newblock Certain extremal problem for nonnegative trigonometric polynomials.
\newblock {\em Mathematical Notes of the Academy of Sciences of the USSR},
  47(1):10–20, January 1990.

\bibitem{ba06}
M.~S. Bazaraa, H.~D. Sherali, and C.~M. Shetty.
\newblock {\em Nonlinear Programming: Theory and Algorithms}.
\newblock Wiley, October 2005.

\bibitem{de99}
C.-J. de~la Vall\'ee~Poussin.
\newblock Sur la fonction $\zeta(s)$ de riemann et le nombres des nombres
  premiers inf\'erieurs \`a une limite donn\'ee.
\newblock {\em M\'em. Couronn\'es Autres M\'em. Publ. Acad. R. Sci. Lett.
  Beaux-Arts Belg.}, 59:1--74, 1899--1900.

\bibitem{fe16}
L.~Fej\'er.
\newblock \"{U}ber trigonometrische polynome.
\newblock {\em crll}, 1916(146):53–82, 1916.

\bibitem{fr66}
S.~H. French.
\newblock Trigonometric polynomials in prime number theory.
\newblock {\em Illinois Journal of Mathematics}, 10(2), June 1966.

\bibitem{ka05}
H.~Kadiri.
\newblock Une région explicite sans zéros pour la fonction $\zeta$ de
  riemann.
\newblock {\em Acta Arithmetica}, 117(4):303–339, 2005.

\bibitem{ko77}
V.~P. Kondrat’ev.
\newblock Some extremal properties of positive trigonometric polynomials.
\newblock {\em Mathematical Notes of the Academy of Sciences of the USSR},
  22(3):696–698, September 1977.

\bibitem{la08}
E.~Landau.
\newblock Beitr\"{a}ge zur analytischen zahlentheorie.
\newblock {\em Rendiconti del Circolo Mathematico di Palermo}, 26(1):169–302,
  December 1908.

\bibitem{mo14}
M.~J. Mossinghoff and T.~S. Trudgian.
\newblock Nonnegative trigonometric polynomials and a zero-free region for the
  riemann zeta-function.
\newblock {\em Journal of Number Theory}, 157:329–349, December 2015.

\bibitem{mo22}
M.~J. Mossinghoff, T.~S. Trudgian, and A.~Yang.
\newblock Explicit zero-free regions for the riemann zeta-function.
\newblock {\em Research in Number Theory}, 10(1), January 2024.

\bibitem{mu00}
J.~R. Munkres.
\newblock {\em Topology}.
\newblock Prentice Hall, 2nd edition, 2000.

\bibitem{po78}
G.~Pólya and G.~Szeg\"{o}.
\newblock {\em Problems and Theorems in Analysis II: Theory of Functions.
  Zeros. Polynomials. Determinants. Number Theory. Geometry}.
\newblock Springer Berlin Heidelberg, 1998.

\bibitem{re86}
A.~V. Reztsov.
\newblock Some extremal properties of nonnegative trigonometric polynomials.
\newblock {\em Mathematical Notes of the Academy of Sciences of the USSR},
  40(3):750–750, September 1986.

\bibitem{ro62}
J.~B. Rosser and L.~Schoenfeld.
\newblock Approximate formulas for some functions of prime numbers.
\newblock {\em Illinois Journal of Mathematics}, 6(1), March 1962.

\bibitem{ro75}
J.~B. Rosser and L.~Schoenfeld.
\newblock Sharper bounds for the chebyshev functions $\theta(x)$ and $\psi(x)$.
\newblock {\em Mathematics of Computation}, 29(129):243–269, 1975.

\bibitem{st70}
S.~B. Stechkin.
\newblock Zeros of the riemann zeta-function.
\newblock {\em Mathematical Notes of the Academy of Sciences of the USSR},
  8(4):706–711, October 1970.

\bibitem{we38}
H.~Westphal.
\newblock {\"U}ber die nullstellen der riemannschen zetafunction im kritischen
  streifen.
\newblock {\em Schr. Math. Semin. Inst. Angew. Math. Univ. Berl.}, 4:1--31,
  1938.

\end{thebibliography}
\end{document}